\documentclass[10pt,a4paper]{article}
\usepackage[utf8]{inputenc}
\usepackage[english]{babel}
\usepackage[T1]{fontenc}
\usepackage{amsmath}
\usepackage{amsfonts}
\usepackage{amssymb}
\usepackage{graphicx}

\usepackage{amsthm}
\usepackage{dsfont}
\usepackage{xcolor}
\usepackage{stix}
\usepackage{bbold}
\usepackage{hyperref}
\usepackage{authblk}

\usepackage[nottoc,numbib]{tocbibind}

\newcommand{\Rq}{\textit{Remark : }}

\newtheorem{theorem}{Theorem}[section]
\newtheorem{lemme}[theorem]{Lemma}
\newtheorem{coro}[theorem]{Corollary}
\newtheorem{propo}[theorem]{Proposition}
\theoremstyle{definition}
\newtheorem{defi}[theorem]{Definition}

\DeclareMathOperator{\HKR}{HKR}
\DeclareMathOperator{\Id}{Id}
\DeclareMathOperator{\H-mod}{\text{H-}mod}
\DeclareMathOperator{\Proj}{Proj}
\DeclareMathOperator{\trace}{\mathtt{t}}
\DeclareMathOperator{\perleTot}{b}
\DeclareMathOperator{\korps}{\mathbb{k}}
\DeclareMathOperator{\categ}{\mathscr{C}}
\DeclareMathOperator{\PtInf}{p_{\infty}}
\DeclareMathOperator{\Vect}{\mathcal{V}ect}
\newcommand{\MatInter}[2]{[ {#1} \cap  {#2} ]}
\DeclareMathOperator{\Kup}{Kup}
\DeclareMathOperator{\SN}{SN}
\DeclareMathOperator{\genre}{\mathsf{g}}

\title{Chromatic spherical invariant and Hennings invariant of $3$-dimensional manifolds}
\author[1]{J. Reina}
\affil[1]{Univ Bretagne Sud, CNRS UMR 6205, LMBA, F-56000, Vannes, France, email : julie.reina@univ-ubs.fr}
\date{}

\begin{document}

\maketitle

\begin{abstract}
This paper establishes a relation between two invariants of $3$-dimensional manifolds: the chromatic spherical invariant $\mathcal{K}$ and the Hennings-Kauffman-Radford invariant $\HKR$. 
We show that, for a spherical Hopf algebra $H$, the invariant $\mathcal{K}$ associated to the pivotal category of finite-dimensional $H$-modules is equal to the invariant $\HKR$ associated to the Drinfeld double $D(H)$ of the same Hopf algebra. 
\end{abstract}

\tableofcontents

\section{Introduction}

It was long conjectured and progressively established until recently (in \cite{Turaev_Virelizier}) that there exists a connection between the Witten–Reshetikhin–Turaev invariants and the Turaev–Viro construction. 
Both of these constructions fundamentally rely on the semi-simplicity of the underlying tensor category.
The Hennings–Kauffman–Radford invariant, denoted $\HKR$, and the Kuperberg invariant, denoted $\Kup$, are often regarded as non-semisimple counterparts of the Witten–Reshetikhin–Turaev invariant and the Turaev–Viro construction, respectively. 
In the non-semisimple setting, it has been analogously conjectured that the Kuperberg invariant $\Kup$, associated to a finite-dimensional Hopf algebra $H$, is related to the invariant $\HKR$ defined using the Drinfeld double of $H$, $D(H)$. 
Since the formulation of this conjecture, several partial results have been established supporting this relation.

The work presented in this paper has been motivated by the article \cite{Chang_Cui}, in which Chang and Cui proved that, for a finite-dimensional double balanced Hopf algebra $H$ and a closed oriented $3$-manifold $M$, there exists a framing $b$ and a 2-framing $\phi$ of $M$ such that
$$ \Kup (M, b; H) = \HKR (M, \phi ; D(H) ). $$

In \cite{CGPT}, Constantino, Geer, Patureau-Mirand, and Turaev introduced the spherical chromatic invariant $\mathcal{K}$ and proved that, for any finite-dimensional involutive pivotal unimodular unibalanced Hopf algebra $H$, there is equality between the invariant $\mathcal{K}_{\H-mod}$ and the Kuperberg invariant associated to $H$ for any connected closed oriented 3-manifold.

In this article, we provide an explicit algorithm for computing the invariant of connected closed oriented $3$-manifold $\mathcal{K}_{\H-mod}$, for any finite-dimensional spherical Hopf algebra $H$. 
To this end, we introduce a new type of diagram, which we call \textit{virtual Heegaard diagrams}. 
These are a variant of classical Heegaard diagrams that allow for crossings on the $\beta$-curves.

This computational method enables us to establish the following restricted version of the conjectured relation mentioned above. 

\begin{theorem}[Main Theorem]
Let $H$ be a spherical Hopf algebra and $M$ be a $3$-dimensional closed connected oriented manifold. 
Then,
$$ \mathcal{K}_{\H-mod} (M) = \HKR_{D(H)}(M), $$
where $D(H)$ denotes the Drinfeld double of $H$ and $\H-mod$ is the pivotal category of finite-dimensional $H$-modules.
\end{theorem}

The invariant $\mathcal{K}_{\categ}$ is defined for a pivotal category $\categ$, and likewise, the $\HKR$ invariant has been generalized for categories by Lyubashenko. 
Both invariants can be viewed as induced by a ribbon TQFT (see \cite{DGGPR1} and \cite{CGPV} respectively).
A similar result to this article was established in \cite{Mull_Schw_Woike_Yang} in the setting of representations of diffeomorphism groups. 
In their work, Müller, Schweigert, Woike, and Yang use string-nets to construct a modular functor $\SN_{\categ}$ from a pivotal category $\categ$ (it is suspected \cite[Rk 2.17]{Mull_Schw_Woike_Yang} that this functor extends into the TQFT of \cite{CGPV}). 
They also prove \cite{Mull_Schw_Woike_Yang} that the functor $\SN_{\categ}$ of a finite pivotal tensor category is equivalent to Lyubashenko’s modular functor associated to the Drinfeld center of that category. 
We suspect the existence of a similar relation between the $(2+1)$-TQFT of \cite{DGGPR1} and that of \cite{CGPV}.

The first part of this article establishes the conventions we will use for Hopf algebras and their Drinfeld doubles. 
We then describe computational methods for the $\HKR$ invariant and for the invariant $\mathcal{K}$.
This will include the definition of virtual Heegaard diagrams and the algorithm for computing $\mathcal{K}_{\H-mod}$ using them. 
In the final part of the paper, we state and prove our main result, establishing the relationship between the two invariants discussed before.

\vspace{15pt}

\textbf{\textit{Acknowledgments : }}
The author wants to thank Bertrand Patureau-Mirand for his guidance and advice throughout the research and writing of this article.
The author also wishes to thank François Constantino and Nathan Geer for their insights that aided in the research for this article. 
This work was partially funded by the France 2030 program, Centre Henri Lebesgue ANR-11-LABX-0020-01.

\section{Hopf algebra's definition}

\subsection{Hopf algebras and integrals}

We begin by establishing the conventions on Hopf algebras that will be used throughout this work.

Let $H = (H, \eta, m, \epsilon, \Delta)$ be a finite-dimensional Hopf algebra over a field $\korps$, with bijective antipode $S$, and unit map $\eta: \korps \to H$ defining the unit element $1_H := \eta(1) \in H$.
For every $n \in \mathbb{N}$, we denote by $\Delta^{(n)} : H \to H^{\otimes n}$ the iteration of the coproduct defined recursively by $\Delta^{(0)} = \epsilon$, $\Delta^{(1)} = \Id_H$ and $\Delta^{(n+1)} = (\Id^{\otimes (n-1)} \otimes \Delta) \circ \Delta^{(n)}$ for $n \geq 1$. 
We will always use Sweedler’s notation for the coproduct,
\begin{equation}
\Delta^{(n)}(x) = x_{(1)} \otimes \cdots \otimes x_{(n)} \label{Not:Sweedler}
\end{equation}
that hides the summation symbol.

A map $f$ will occasionally be written as $f(\varspadesuit)$, where $\varspadesuit$ denotes the argument. 
We denote by $\tau : H \otimes H \to H \otimes H$ the map $ : a \otimes b \mapsto b \otimes a$.

We recall some standard definitions on Hopf algebras (see \cite{RadHA} and \cite{GPP} for more details).

\begin{defi}
\begin{itemize}
\item A right (resp. left) \emph{integral} on $H$ is an element $\lambda \in H^*$ such that
$$ \lambda(x_{(1)})x_{(2)} = \lambda(x) 1_H \qquad (\text{resp. } x_{(1)}\lambda(x_{(2)}) = \lambda(x) 1_H), $$
for all $x \in H$.

\item A left (resp. right) \emph{cointegral} in $H$ is an element $\Lambda \in H$ such that
$$ x \Lambda = \epsilon(x) \Lambda \qquad (\text{resp. } \Lambda x  = \epsilon(x) \Lambda), $$
for all $x \in H$.

\item A \emph{group-like} element in $H$ is an element $g \in H$ satisfying $\Delta (g) = g \otimes g$ and $\epsilon(g) = 1$.
\end{itemize}
\end{defi}

It follows from the definition that a group-like element $g$ is invertible and satisfies $S(g) = g^{-1}$.

\begin{defi}

\begin{itemize}
\item Let $\Lambda$ be a nonzero left cointegral of $H$. 
There exists a group-like element $\alpha \in H^*$ determined by the relation  
$$ \Lambda v = \alpha(v) \Lambda, $$
for all $v \in H$. 
We call $\alpha$ the \emph{$H$-distinguished group-like element of $H^*$}. 

\item Let $\lambda$ be a nonzero right integral of $H$.
There exists a group-like element  $a \in H$ determined by the relation
$$ f \lambda = f(a) \lambda, $$
for all $f \in H^*$. 
We call $a$ the \emph{$H$-distinguished group-like element of $H$}. 
\end{itemize}
\end{defi}

Since $H$ is finite-dimensional, there is a unique (up to a scalar) right integral and left cointegral.
Therefore, the distinguished group-like elements $\alpha$ and $a$ are independent from the choice of $\Lambda$ and $\lambda$. 
From now on, we fix a right integral $\lambda \in H^*$ and a left cointegral $\Lambda \in H$ such that $\lambda(\Lambda)=1$.

\begin{defi}
A Hopf algebra $H$ is said to be \emph{unimodular} if it is finite-dimensional and if left cointegrals are also right cointegrals.
\end{defi}

\Rq Saying that $H$ is unimodular is equivalent to having $\alpha = \epsilon$. 
\vspace{5pt} 

\Rq \cite{RadHA}  If $H$ is unimodular, then $\lambda$ is a quantum character, i.e., $\lambda(xy) = \lambda(S^2(y)x)$ for all $x, y \in H$. 
In particular, $\lambda \circ S^2 = \lambda$. 
\vspace{5pt}

\begin{defi}
\begin{itemize}
\item A Hopf algebra $H$ with antipode $S$ is called \emph{pivotal} with pivot $g \in H$ if $g$ is a group-like element such that $ \forall x \in H, S^2(x) = gxg^{-1}$. 

\item A Hopf algebra is called \emph{spherical} if it is pivotal and unimodular (hence finite-dimensional), with its pivot $g$ satisfying $g^2 = a$. 
\end{itemize}

\end{defi}

\Rq In the spherical case, we have the equality $\lambda \circ S = \lambda (g^2 \varspadesuit)$ (direct consequence of \cite[Th. 10.5.4]{RadHA}). 
\vspace{5pt}

\begin{defi}
A \emph{symmetrized integral} in a pivotal Hopf algebra $H$ is a linear form $\mu$ satisfying:
\begin{enumerate}
\item $(\mu \otimes g)\Delta(x) = \mu(x)1_H$ for all $ x \in H$,
\item $\mu(xy) = \mu(yx)$ for all $x, y \in H$,
\item $\mu(S(x))= \mu(x)$ for all $x \in H$. 
\end{enumerate}
\end{defi}

\begin{propo}
Let $H$ be a spherical Hopf algebra and $\lambda$ a right integral. 
 Then the linear form $\mu \in H^*$ defined by $\mu := \lambda(g \varspadesuit)$ is a symmetrized integral.
\end{propo}

\begin{proof}
 
For all $x \in H$, by the definition of a right integral, we have
$$ \mu(x)1_H = \lambda(gx) 1_H = \lambda((gx)_{(1)})(gx)_{(2)} = \lambda(gx_{(1)}) gx_{(2)} = (\mu \otimes g) \Delta(x). $$
Thus, equality $1.$ is verified. 

In the unimodular case, by properties of right integrals and the definition of $g$ as the pivot, for all $x, y \in H$ we have
$$ \mu(xy) = \lambda(gxy) = \lambda(S^2(y)gx) = \lambda(gyx) = \mu(yx) $$
Thus, equality $2.$ is verified. 

Since $g$ is group-like, in particular $S(g) = g^{-1}$ and $S(g^{-1}) = g$. 
Therefore, for all $x \in H$, we have
$$ \mu(S(x)) = \lambda(gS(x)) = \lambda(S(xg^{-1})) = \lambda(g^2 x g^{-1}) = \lambda(gx) = \mu(x), $$
due to the properties of right integrals in the unimodular and spherical case.
Thus, equality $3.$ is verified.  
\end{proof}

From now on, we will consider the symmetrized integral defined by $\mu:= \lambda(g \varspadesuit)$.  

\begin{defi}
A Hopf algebra $H$ is called \emph{quasitriangular} if it admits an invertible element $R \in H \otimes H$, denoted $R = \sum^n_{i=1} r_i \otimes s_i$, satisfying
\begin{enumerate}
\item $\sum^n_{i=1} \Delta(r_i) \otimes s_i = \sum^n_{i,j=1} r_i \otimes r_j \otimes s_i s_j $,
\item $\sum^n_{i=1} r_i \otimes \Delta(s_i) = \sum^n_{i,j=1} r_j r_i \otimes s_i \otimes s_j, $
\item $ \tau \circ \Delta(h) R = R \Delta(h).$
\end{enumerate}
This element $R$ is called the $R$-matrix of $H$. 
To alleviate the notations, we will sometimes leave the summation implicit and write $R = r \otimes s$ (or $R= r_i \otimes s_i$ to keep track of the pairing between $r$ and $s$ when multiple iterations of $R$ appear). 
\end{defi}

The element $u := \sum_{i=1}^n S(s_i) r_i \in H$ is called the Drinfeld element.
It is invertible and satisfies, for all $h \in H$, $S^2(h) = u h u^{-1}$ and $uS(u) = S(u)u \in Z(H)$. 

\begin{defi}
A quasitriangular Hopf algebra $H$ is called \emph{ribbon} if it admits an element $\theta \in Z(H)$, called the ribbon element, satisfying the following equations
$$ \theta^{-2} = uS(u), \qquad S(\theta) = \theta, \qquad \epsilon(\theta) = 1, \qquad \Delta(\theta) = (\sum_{i,j} s_j r_i \otimes r_j s_i)(\theta \otimes \theta). $$ 
The algebra is then pivotal with pivot element $g = u \theta$ (more details in the proof of \cite[Prop 12.3.6]{RadHA}).
\end{defi}

An equivalent definition in the case of a pivotal structure is given as follows (see \cite{GPP}). 

\begin{defi}[Equivalent definition, \cite{GPP}]
A quasitriangular pivotal Hopf algebra $H$ is called \emph{ribbon} if the element $\theta := m(\tau((g \otimes 1_H)R))$ satisfies
$$ \theta = m((1_H \otimes g^{-1})R).$$
\end{defi}

\Rq Since $(S \otimes S)(R) = R$, this definition is equivalent to requiring that $\theta$ satisfies  $S(\theta) = \theta$. 
\vspace{5pt}

\begin{propo} \label{PivotalEnrubanné}
An unimodular, quasitriangular, pivotal Hopf algebra $H$ is ribbon if and only if it is spherical.
\end{propo}

\begin{proof}
If a Hopf algebra $H$ is unimodular, then $a = u S(u^{-1})$, with $a$ the distinguished group-like element of $H$ (see \cite[Th 12.3.3.b]{RadHA}).

Let us define $\theta := m(\tau((g \otimes 1_H)R))$. 
Then it can be seen that $u= \theta^{-1} g$, and thus, that $\theta \in Z(H)$. 

Let us assume that $H$ is ribbon. 
Since $\theta = S(\theta)$ and $g$ is a pivot for $H$, we compute 
$$a = uS(u^{-1}) = S(u^{-1})u = \theta g  \theta^{-1} g = g \theta \theta^{-1} g = g^2. $$

Thus, $H$ is spherical. 

Conversely, let us assume that $H$ is spherical. 
Then, 
\begin{align*}
& \qquad u S(u^{-1}) = \theta^{-1} g^2 S(\theta) \\
\iff &\qquad a = \theta^{-1} g^2 S(\theta) \\
\iff &\qquad \theta g^2 = g^2 S(\theta) \\
\iff &\qquad g^2 \theta = g^2 S(\theta) \\
\iff &\qquad \theta = S(\theta) 
\end{align*}

Thus, $H$ is ribbon. 
\end{proof}

\textit{Example of a spherical Hopf algebra:}

The small quantum group $U_q = \overline{U_q(\mathfrak{sl}(2))}$ (with $q$ a $2r$-th root of unity) is the algebra generated by $E, F, K, K^{-1}$ together with the relations:
$$ K^r = 1, \quad E^r = F^r = 0, \quad (r \in \mathbb{N}) $$
$$ K K^{-1} = K^{-1} K = 1, $$
$$  KE = q^2 EK,  \quad \quad KF = q^{-2} FK,$$
$$ [E,F] = \frac{K - K^{-1}}{q- q^{-1}} $$

The following theorem gives its structure as a vector space. 
(A demonstration of this theorem and more details about $U_q$ can be found in \cite{Kassel}.)

\begin{theorem}[Poincaré-Birkhoff-Witt] 
\hfill

    The set $\{ K^m E^n F^l\}_{ m,n,l \in \{0, ..., r-1\}}$ is a basis of $U_q$. 
\end{theorem}

The following relations give $U_q$ a structure as a Hopf algebra. 
$$ \Delta (K)= K \otimes K, \quad \Delta(K^{-1}) = K^{-1} \otimes K^{-1} $$
$$ \Delta (E) = 1 \otimes E + E \otimes K, \quad \Delta (F) = K^{-1} \otimes F + F \otimes 1, $$
$$ \epsilon(E) = \epsilon (F) = 0, \quad \epsilon(K) = \epsilon(K^{-1}) = 1, $$
$$ S(K) = K^{-1}, \quad S(K^{-1})= K, \quad S(E) = -E K^{-1}, \quad S(F) = -K F. $$

Furthermore, it can be seen that $S^2(x) = K x K^{-1}$, thus, $g = K$ can be chosen as pivot of $U_q$. 

Let $c \in \korps$ be an invertible constant, a right integral of $U_q$ is defined, on the previously mentioned basis of $U_q$, by 
    $$ \lambda(K E^{r-1} F^{r-1})  = \frac{r}{c}, \quad \text{ and } \quad \lambda(K^m E^n F^l) = 0 \text{ si }(m,n,l) \not =  (1, r-1, r-1) $$
and a cointegral $\Lambda$ is given by
$$ \Lambda = c(\frac{1}{r} \sum^{r-1}_{j=0} K^j)E^{r-1}F^{r-1}. $$
We can note that $\lambda(\Lambda) = 1$ (more details in \cite{CGHP}).

From $\lambda$, we get the symmetrized integral $\mu$, defined by
  $$ \mu(E^{r-1} F^{r-1})  = \frac{r}{c}, \quad \text{ and } \quad \mu(K^m E^n F^l) = 0 \text{ si }(m,n,l) \not =  (0, r-1, r-1).$$

\subsection{Drinfeld double}\label{PartDoubleDrinfeld}

We consider the finite-dimensional Hopf algebra $H^{*,\mathrm{cop}}$ defined from $H$ by $H^{*,\mathrm{cop}} := (H^*)^{\mathrm{cop}} = (H^*, { }^t\epsilon, { }^t\Delta, { }^t\eta, { }^tm^{\mathrm{op}})$ with antipode ${ }^t(S^{-1})$.
Here, ${ }^tm^{\mathrm{op}}$ is the map defined by $(m^{\mathrm{op}}(f))(x \otimes y) = f(yx),$ for all $f \in H^*$ and all $x, y \in H$. 
The map ${ }^t \Delta : H^* \otimes H^* \to H^*$ is defined through the identification $(H \otimes H)^*=H^* \otimes H^*$ by $({ }^t \Delta(f \otimes g))(x) = f(x_{(1)}) \otimes g(x_{(2)})$, for all $x \in H$ and all $f, g \in H^*$. 

\emph{Structure on $D(H)$ : }
Following the conventions of \cite{Chang_Cui} and \cite{RadHA}, the Drinfeld double $D(H) = H^{*\mathrm{cop}} \otimes H$ of a finite-dimensional Hopf algebra $H$ is a quasitriangular Hopf algebra with the structure given below.

For all $f \in (H^*)^{\mathrm{cop}}$ and all $v \in H$, we have

$$ 1^D = \epsilon \otimes 1_H, \qquad \epsilon^D = 1_H \otimes \epsilon, \qquad \Delta^D = (f_{(1)} \otimes x_{(1)}) \otimes (f_{(2)} \otimes x_{(2)}), $$
$$ M^D(f \otimes v, f' \otimes v') = (f f'(S^{-1}(v_{(3)}) \varspadesuit v_{(1)}) \circ \Delta \otimes v_{(2)}v',$$
$$ S^D(f \otimes v) = M^D(\epsilon \otimes S(v) , f \circ S^{-1} \otimes 1) = f \circ S^{-1}(v_{(3)} \varspadesuit S(v_{(1)})) \otimes S(v_{(2)}).$$

Let $(e_1, ..., e_n)$ be a basis of $H$ and $(e^*_1, ..., e^*_n)$ the dual basis of $H^{*,\mathrm{cop}}$, we can make explicit the $R$-matrix of $D(H)$ by
$$ R^D := \sum^n_{i = 1} (\epsilon \otimes e_i) \otimes (e^*_i \otimes 1_H) $$
and its inverse by 
$$ (R^D)^{-1} := \sum^n_{i = 1} (\epsilon \otimes S(e_i)) \otimes (e^*_i \otimes 1_H). $$

From now on, we denote the multiplication in the double by $(f \otimes v)( f' \otimes v') := M^D(f \otimes v, f' \otimes v')$. 
It can be noted, by direct calculation, that $(f \otimes 1_H)( f' \otimes v') = f f' \otimes v'$ and $(f \otimes v)( \epsilon \otimes v') = f \otimes v v'$, for all $(f \otimes v), ( f' \otimes v') \in D(H)$.

\begin{propo}\cite[Th 13.2.1]{RadHA} 
The Drinfeld double $D(H) = H^{* \mathrm{cop}} \otimes H$ of a finite-dimensional Hopf algebra $H$ is an unimodular Hopf algebra (for it is both quasitriangular and factorizable).
\end{propo}

As with the structure maps of $D(H)$, we can define a right integral and a $D(H)$-distinguished group-like element based on those previously established on $H$, 
$$ \lambda^D := \Lambda \otimes \lambda \qquad \text{ and } \qquad a^D = \alpha \otimes a,$$
(see \cite[Part. 2.4]{Chang_Cui} and \cite[Part 13.7]{RadHA}).

\begin{lemme}\cite[Lemma 13.7.2]{RadHA}\label{Lemmeg_D}
Let $H$ be a finite-dimensional Hopf algebra over a field $\korps$, and let $a$ and $\alpha$ be the $H$-distinguished and $H^*$-distinguished group-like elements, respectively. 
Then, for any pair $(\beta , b) \in H^* \times H$ such that
\begin{itemize}
\item $\beta$ and $b$ are group-like elements, 
\item $\beta^{-2} = \alpha$, $b^{-2} = a$, and 
\item $S^2(h)= b^{-1}(\beta(h_{(1)})h_{(2)} \beta^{-1}(h_{(3)}))b $, for all $h \in H$,
\end{itemize} 
the element $(\beta \otimes b)^{-1}$ is a pivot of $D(H)$.  
\end{lemme}

When $H$ is spherical (and so, unimodular), we have $\alpha= \epsilon$ and $g^2 = a$. 
Therefore, the following proposition can be deduced from the lemma. 

\begin{propo}\label{g_D}
Let $H$ be a spherical Hopf algebra with $g$ as its pivot. 
Then its Drinfeld double $D(H)$ is also spherical with $g^D:= \epsilon \otimes g$ as its pivot. 
\end{propo}

\begin{proof}
First, let us show that the pair $(\epsilon, g^{-1})$ satisfies the assumptions of Lemma \ref{Lemmeg_D}. 

By definition, $g$ and $\epsilon$ are group-like, therefore, $g^{-1}$ also is. 
Since $H$ is spherical and unimodular, we have that $\alpha= \epsilon$ and $g^2 = a$. 

It remains to verify that $S^2(h)= g(\epsilon^{-1}(h_{(1)})h_{(2)} \epsilon(h_{(3)}))g^{-1} $ for all $h \in H$.
However, by the definition of $\epsilon$ and of the pivot $g$, we have
$$ g(\epsilon^{-1}(h_{(1)})h_{(2)} \epsilon(h_{(3)}))g^{-1} = ghg^{-1} = S^2(h),$$ 
for all $h \in H$.
Therefore, $g^D=(\epsilon \otimes g^{-1})^{-1}$ can be set as pivot of $D(H)$.

As mentioned earlier, by direct calculation we have $(\epsilon \otimes g^{-1}) (\epsilon \otimes g) = \epsilon \otimes 1_H$. 
Therefore $(\epsilon \otimes g^{-1})^{-1} = \epsilon \otimes g$. 

Furthermore, since $H$ is spherical, we note that 
$$ (g^D)^2 = (\epsilon \otimes g)(\epsilon \otimes g) = \epsilon \otimes g^2 = \epsilon \otimes a = a^D. $$ 
Thus, $D(H)$ is also spherical.
\end{proof}

We can argue similarly that, if a pair $(\beta , b) \in H^* \times H$ as described in Lemma \ref{Lemmeg_D} exists, then $D(H)$ with the pivot $\beta^{-1} \otimes b^{-1}$ is spherical. 

\vspace{5pt}
\Rq Proposition \ref{g_D} allows us to express a symmetrized integral of $D(H)$ through elements of $H$ and $H^*$. 
$$ \mu^D = \lambda^D(g^D \varspadesuit) = \Lambda \otimes \mu $$

\Rq Since the Drinfeld double of a Hopf algebra is unimodular and quasitriangular (see \cite[Th 13.2.1]{RadHA}), Propositions \ref{PivotalEnrubanné} and \ref{g_D} imply that the double of a spherical Hopf algebra is also ribbon.

\section{Computation of the $\HKR$ invariant}\label{Part:CalcHenn}

Let $M$ be a $3$-dimensional manifold and $L$ an oriented surgery link of $M$. 
Let $H$ be a finite-dimensional ribbon Hopf algebra with pivot $g$, and $\mu$ a symmetrized non-degenerate integral of $H$ (i.e. $\mu(g \theta)\mu(g^{-1} \theta^{-1}) \not = 0 \in \korps$). 
We denote $\delta := \mu ( g \theta) \in \korps$. 
In the rest of this article, we assume that $\mu$ is renormalized such that $\mu(g \theta)\mu(g^{-1} \theta^{-1}) =1 \in \korps$. 
Therefore, we will use the $\HKR$ invariant with the computational algorithm described in \cite{GPP} that we express as follows. 

\begin{itemize}
\item Let $D$ be a planar diagram of $L$ with each component admitting a base point.
For each cap, cup, and crossing, we place on $D$ the beads described in Figure \ref{fig:FoncPerleGPP}.
Each bead is colored by an element of the algebra $H$. 
Furthermore, we denote $R = r \otimes s$.
\begin{figure}[hbtp]
\centering
\includegraphics[scale=0.7]{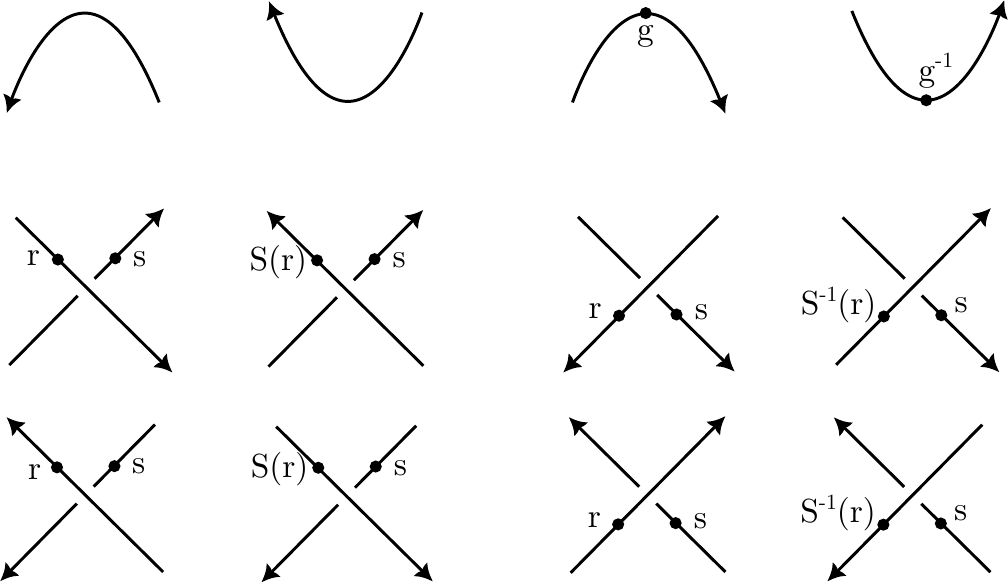}
\caption{Beads placement (according to the conventions of \cite{GPP}) }\label{fig:FoncPerleGPP}
\end{figure}

\item Starting from the base point and following the orientation, we multiply all beads in one element, which gives an element of $H$ for each component. 
(All multiplications are performed from the left.) 
This produces an element $PT(D) \in H^{\otimes n}$ (with $n$ the number of components of $L$) which we call the total bead of $D$ (at this point, it is not an invariant of the diagram).

\item The $\HKR$ invariant is given by the following theorem. 
\end{itemize}

\begin{theorem} \cite[Th 2.9]{GPP}
Let $H$ be a finite-dimensional Hopf algebra. 
For each manifold $M$, the scalar $ \HKR_H(M) $ is a manifold invariant defined as 
$$ \HKR_H(M) := \delta^{-s} \mu^{\otimes n}(PT(D)) \in \korps $$
where $D$ is a diagram of a surgery link of $M$, $s \in \mathbb{Z}$ is the signature of the linking matrix of $L$ and $n$ is the number of components of $L$. 

\end{theorem}

This theorem provides a reformulation of the 3-manifold invariant introduced by Hennings \cite{Hen96} and by Kauffman and Radford \cite{KR95}. 
A proof can be found in \cite[Th 2.9]{GPP} considering $G = \{ 0 \}$.  

\Rq The beads on the rightmost crossing in the Figure \ref{fig:FoncPerleGPP} are denoted $ r_i$ and $S(s_i)$ instead of $S^{-1}(r_i)$ and $s_i$ as in other references (see \cite[Part. 2.2]{GPP}).
Both conventions are equivalent because $(S \otimes S) (R) = R$. 
\vspace{5pt}

\section{The chromatic spherical invariant $\mathcal{K}$}

In this section, we fix $H$ to be a spherical Hopf algebra over a field $\korps$. 

\subsection{Virtual Heegaard diagrams}

Next, we will define what we will call \emph{virtual Heegaard diagrams}. 
Roughly speaking, these are Heegaard diagrams allowing intersection points over the $\beta$ curves. 

We assume that all mentioned curve are oriented. 

\begin{defi}[Heegaard diagrams]
We call \emph{Heegaard diagram} a triplet $D =(\Sigma_{\genre}, \alpha, \beta)$ such that

\begin{itemize}
\item $\Sigma_{\genre}$ is a closed oriented surface of genus ${\genre} \in \mathbb{N}$. 
\item $\alpha = (\alpha_1, ..., \alpha_{\genre})$ is a collection of ${\genre}$ simple closed disjoint curves.
\item $\beta = (\beta_1, ..., \beta_{\genre})$ is a collection of ${\genre}$ simple closed disjoint curves.
\item The complement of $\alpha$ and the complement of $\beta$ in $\Sigma_{\genre}$ are connected. 
\item Every intersection between two curves is transverse. 
\end{itemize}
\end{defi}

In this article, we choose to consider the curves in a Heegaard diagram as oriented, although this is not usually the case in the literature on Heegaard diagrams. 

Every Heegaard diagram defines a $3$-manifold by gluing together two handlebodies determined by the sets of curves $\alpha$ and $\beta$, with boundary the surface $\Sigma_{\genre}$ (more details in \cite[Part 2]{Ozsvath_Szabo}).  
Our interest in these diagrams stems from the following classical result (see \cite[Theorem 2.1 and Remark 2.6]{Ozsvath_Szabo}).

\begin{theorem} 
Every oriented closed $3$-manifold admits a Heegaard diagram.
\end{theorem}

In this article, we define and use a more general and abstract type of diagram encompassing Heegaard diagrams.

\begin{defi}[Virtual Heegaard diagrams]
We denote by $D^v_{He}$ the set of triplets $D =(\Sigma_{\genre}, \alpha, \beta)$ such that 

\begin{itemize}
\item $\Sigma_{\genre}$ is a closed oriented surface of genus ${\genre} \in \mathbb{N}$. 
\item $\alpha = (\alpha_1, ..., \alpha_{\genre})$ is a collection of ${\genre}$ simple closed disjoint curves.
\item $\beta = (\beta_1, ..., \beta_{\genre})$ is a collection of ${\genre}$ closed immersed curves.
\item The complement of $\alpha$ in $\Sigma_{\genre}$ is connected. 
\item  Every (self-)intersection of the curves is transverse. 
\end{itemize}

We call \emph{virtual Heegaard diagrams} the elements of $D^v_{He}$.

\end{defi}

\Rq It is clear that a regular Heegaard diagram is also a virtual Heegaard diagram. 
\vspace{5pt}

\begin{figure}[ht]
\centering
\includegraphics[scale=0.75]{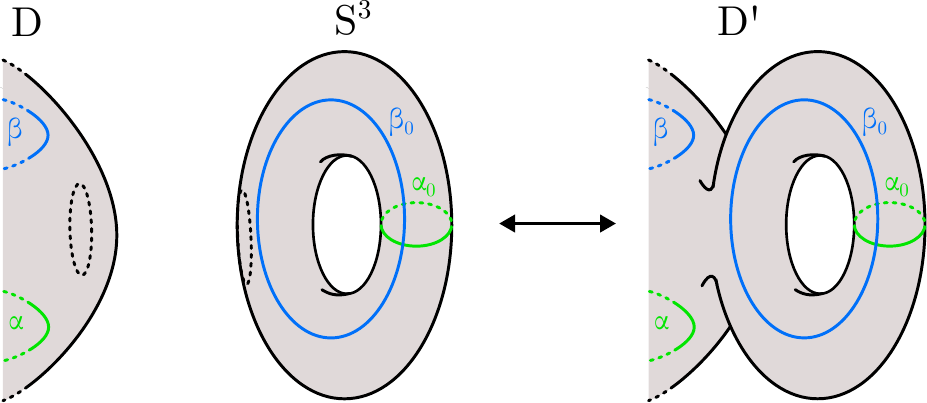}
\caption{Stabilization on a diagram $D$ or destabilization on a diagram $D'$}\label{fig:Stabilisation}
\end{figure}

\begin{defi}

Let $D = (\Sigma_{\genre}, \alpha, \beta) \in D^v_{He}$ be a virtual Heegaard diagram. 
Let us consider the Heegaard diagram of genus $1$ of the sphere $S^3$, $(T, \alpha_0, \beta_0)$ (as represented in Figure \ref{fig:Stabilisation}). 
Let $\Sigma_{\genre}'$ be a connected sum of $\Sigma_{\genre}$ with the torus $T$ along a disk disjoint from the sets of curves $\alpha, \alpha_0, \beta$ and $\beta_0$. 
Define $\alpha' = \alpha_0 \cup \alpha$ and $\beta' = \beta_0 \cup \beta$. 
We say that the virtual Heegaard diagram $D' =  (\Sigma_{\genre}', \alpha', \beta')$ is a \emph{stabilization} of $D$.

Conversely, we say that $D$ is a \emph{destabilization} of $D'$. 
\end{defi}

\begin{figure}[ht]
\centering
\includegraphics[scale=0.5]{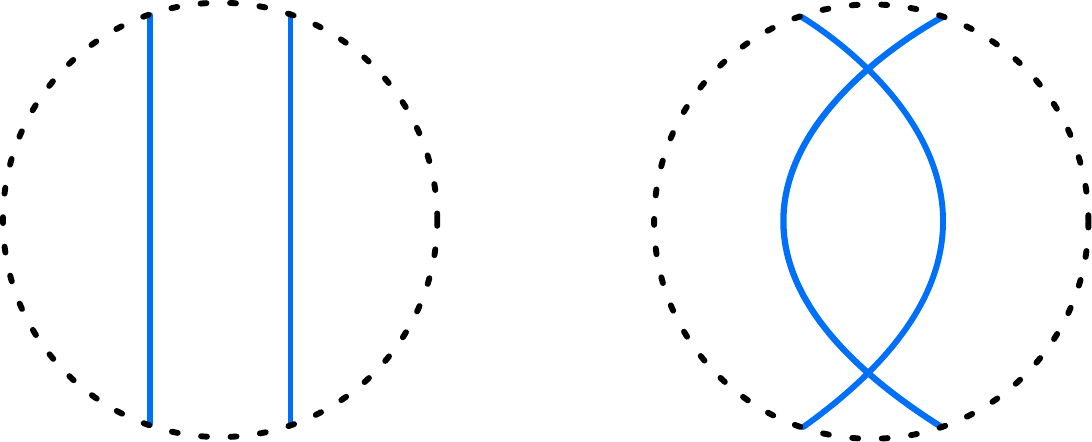}
\caption{Diagrams differing by a $R_{II}$ move (here, between $\beta$ curves).}\label{fig:MouvR2}
\end{figure}

\begin{defi}
We say that two Heegaard diagrams are equivalent if we can link one to the other through a sequence of isotopies of the surface, handle-slides (of $\beta$ curves or $\alpha$ curves between themselves), stabilizations and destabilizations, type $2$ Reidemeister moves between $\beta$ and $\alpha$ curves (also called $R_{II}$ moves, see Figure \ref{fig:MouvR2}) and orientation changes on $\alpha$ and $\beta$ curves. 
\end{defi}

\begin{theorem}
Two Heegaard diagrams representing the same manifold are equivalent. 
\end{theorem}

\Rq This Theorem (more details in \cite[Theorem 2.9]{Ozsvath_Szabo}) can be seen as a direct consequence or a Corollary of Reidemeister\cite{Reidemeister}-Singer\cite{Singer}'s Theorem, for which proof can be found in \cite{Laudenbach}. 

\begin{figure}[ht]
\centering
\includegraphics[scale=0.5]{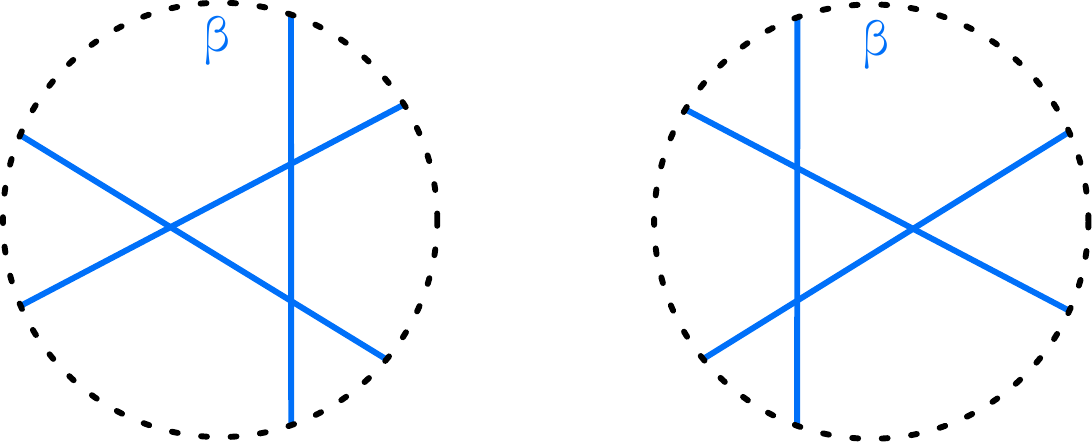}
\caption{$R_{III}$ move}\label{fig:MouvR_3}
\end{figure}

On virtual Heegaard diagrams, we also consider the following moves : 
\begin{itemize}
\item Handle slides of $\beta$ curves over $\alpha$ curves, 
\item Type $3$ Reidemeister moves between $\beta$ curves (also called $R_{III}$ moves (see Figure \ref{fig:MouvR_3}), 
\end{itemize}

\Rq In the rest of this article, we use the diagram representations of \cite{Chang_Cui} and described below. 
\vspace{5pt}

\begin{figure}[ht]
\centering
\includegraphics[scale=0.13]{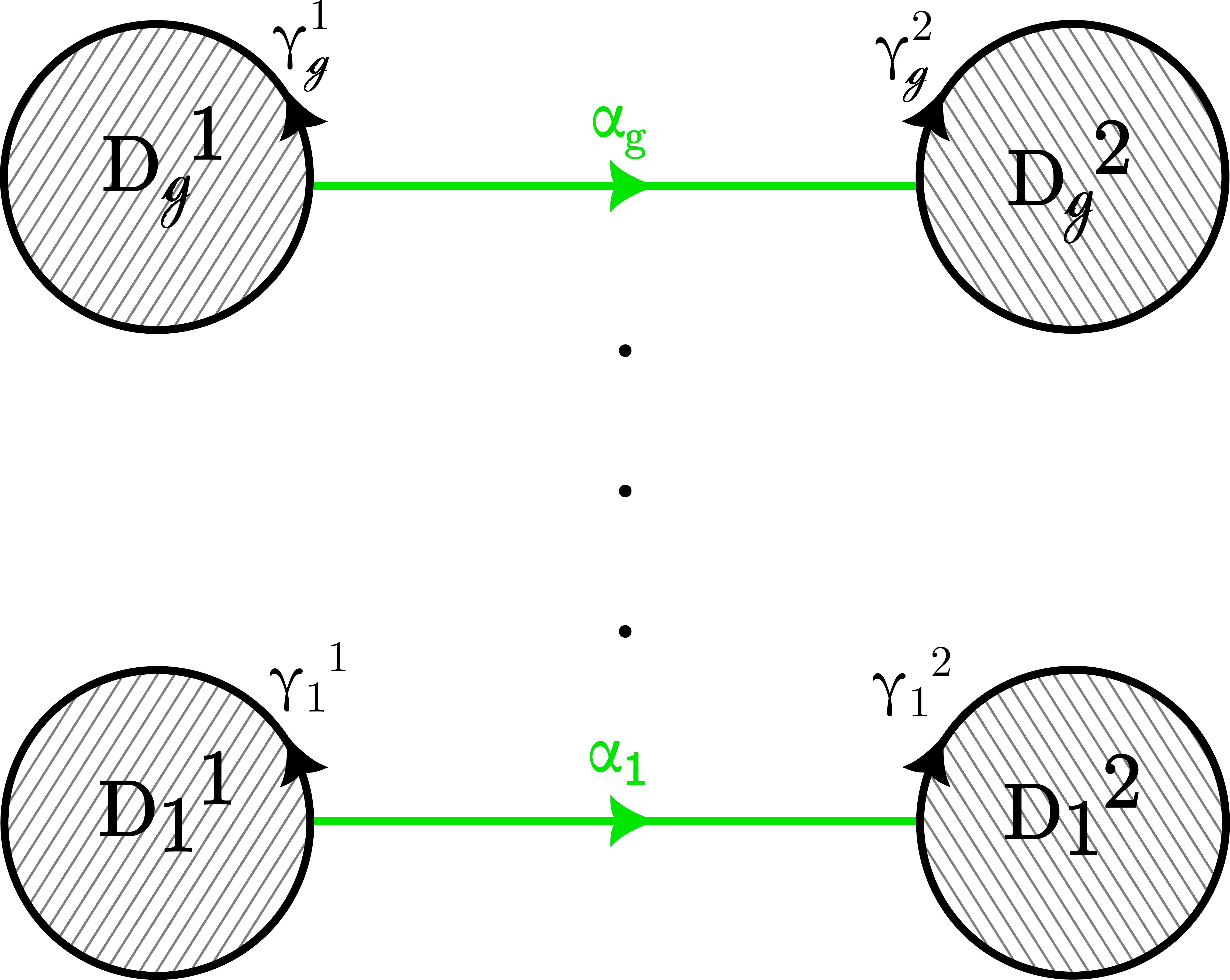} 
\caption{Genus ${\genre}$ Heegaard diagram of $\mathbb{S}^3$}\label{fig:DiagHeegType}
\end{figure}

Let $\genre$ be a natural integer. 
We consider, in $\mathbb{R}^2$, $\genre$ pairs of disks $(D_i^1, D_i^2)_{1 \leq i \leq g} \subset \mathbb{R}^2$ of the same radius $< \frac{1}{3}$ and with centers $(1,i)$ and $(2,i)$, respectively.
For each pair $(D_i^1, D_i^2)$, we consider $\alpha_i$ the minimum length segment connecting $D_i^1$ and $D_i^2$, oriented from $D_i^1$ to $D_i^2$. 
We see the boundaries of these disks as curves $\gamma_i^1$ and $\gamma_i^2$, oriented such that their intersections with $\alpha_i$ verify $\alpha_i \cdot \gamma_i^1 = \alpha_i \cdot \gamma_i^2 = +1$ (see Figure \ref{fig:DiagHeegType}).

We define the \emph{geometric intersection matrix} of two families of ${\genre}$ curves, $\alpha$ and $\gamma$, with transverse intersections, as $\MatInter{\alpha}{\gamma} := \big(\# (\alpha_i \cap \gamma_j)\big)_{i,j \in \{1, ..., {\genre} \}}$. 

\begin{defi}
We say that a Heegaard diagram $(\Sigma_{\genre}, \alpha, \gamma)$ of $\mathbb{S}^3$ is \emph{standard} if $\MatInter{\alpha}{\gamma}= \Id_{\genre}$ and if all intersections between $\alpha$ and $\gamma$ are positive. 

Let $(\Sigma_{\genre}, \alpha', \gamma')$ be a standard diagram of $\mathbb{S}^3$, and let $\PtInf$ be a point on the surface $\Sigma_{\genre} \setminus (\alpha' \cup \gamma')$. 
For all $i \in \{1, ..., {\genre}\}$, let $V_i = \gamma'_i \times [-\varepsilon, \varepsilon]$ (with $\varepsilon > 0$) be a tubular neighborhood of $\gamma'_i$ with its two boundaries denoted by $\gamma'^1_i = \gamma'_i \times -\{\varepsilon\}$ and $\gamma'^2_i = \gamma'_i \times \{ \varepsilon \}$. 

A \emph{standard projection} is an orientation preserving diffeomorphism
$$\psi : \Sigma_{\genre} \setminus \big( \{ \PtInf \} \cup \underset{i \in \{1, ..., {\genre}\}}{\bigcup} \mathring{V_i} \big) \to \mathbb{R}^2  \setminus \underset{i \in \{1, ..., {\genre}\}, j \in \{1,2\}}{\bigcup} \mathring{D_i^j} $$
which sends the set of curves $\{\alpha'_i \}_{i \in \{1, ..., {\genre}\}}$ over the previously described set of curves $\{ \alpha \}_{i \in \{1, ..., {\genre}\}}$. 
This induces that the boundaries $\{ \gamma'^1_i \}_{i \in \{1, ..., {\genre}\}}$ and $ \{ \gamma'^2_i \}_{i \in \{1, ..., {\genre}\}}$ of $\Sigma_{\genre} \setminus \big( \{ \PtInf \} \cup \underset{i \in \{1, ..., {\genre}\}}{\bigcup} \mathring{V_i} \big)$ are mapped onto the boundaries $\{ \gamma^1_i \}_{i \in \{1, ..., {\genre}\}}$ and $ \{ \gamma^2_i \}_{i \in \{1, ..., {\genre}\}}$ of $\mathbb{R}^2  \setminus \underset{i \in \{1, ..., {\genre}\}, j \in \{1,2\}}{\bigcup} \mathring{D_i^j}$, respectively. 
\end{defi}

Let $\phi$ be the index permutation defined by $\psi(\alpha'_i) = \alpha_{\phi(i)}$ for all $i \in \{1, ..., {\genre}\}$.
This permutation is also the one permuting the projections of the indices of $\gamma'^1$ and $\gamma'^2$. 
Therefore, we find that $\psi((\alpha'_i, \gamma'^1_i, \gamma'^2_i)) = (\alpha_{\phi(i)}, \gamma^1_{\phi(i)}, \gamma^2_{\phi(i)})$, for all $i \in \{1, ..., {\genre}\}$. 

We can make the identification $\Sigma \setminus \gamma' \simeq \Sigma \setminus \overline{V_i}$. 
Through that identification, $\gamma'^1_i$ and $\gamma'^2_i$ represent the same curve $\gamma'_i$ of $\Sigma_{\genre}$. 
In the rest of this article, we will also use $\gamma_{\phi(i)}$ for such images obtained from a standard projection of curves $\gamma'_i$. 

\Rq In \cite{Chang_Cui} the notion of standard diagrams also appears, however, curves are non-oriented in that article. 

\begin{defi}\label{def:DiagStandardisé}
Let $D = (\Sigma_{\genre}, \alpha, \beta)$ be a virtual Heegaard diagram along with a fixed point $\PtInf \in \Sigma_{\genre} \setminus (\alpha, \beta)$. 
Let us consider a family of curves $\gamma \subset \Sigma_{\genre} \setminus \{ \PtInf \}$ such that $(\Sigma_{\genre}, \alpha, \gamma)$ is standard and that the intersections $\beta \cap \gamma$ are transverse. 
We then say that $(\Sigma_{\genre}, \alpha, \beta, \PtInf, \gamma)$ is a \emph{standardized} diagram. 
\end{defi}

\Rq For all virtual Heegaard diagrams $(\Sigma_{\genre}, \alpha, \beta)$ and all fixed points $\PtInf \in \Sigma_{\genre} \setminus (\alpha, \beta)$, there exists at least a family $\gamma \subset (\Sigma_{\genre} \setminus \{\PtInf\})$ of ${\genre}$ simple disjoint closed curves such that $\MatInter{\alpha}{\gamma} = \Id_{\genre}$. 

Indeed, the surface $\Sigma_{\genre}$ with the $\alpha$ curves removed is diffeomorphic to an open disk with $2{\genre}-1$ holes. 
We can then establish on that disk a family $\gamma \subset \Sigma_{\genre}$ of ${\genre}$ simple disjoint curves with connected complement, such that $\MatInter{\alpha}{\gamma} = \Id_{\genre}$. 

Up to changes on the orientation of some curves of $\gamma$, we can deduce from this that there always exists a family of curves $\gamma$ for which $(\Sigma_{\genre}, \alpha, \gamma)$ is standard. 
Therefore, for all virtual Heegaard diagrams $(\Sigma_{\genre}, \alpha, \beta)$, there exists at least one standardized virtual Heegaard diagram $(\Sigma_{\genre}, \alpha, \beta,\allowbreak \PtInf, \gamma)$.

\begin{defi}
Let $(\Sigma_{\genre}, \alpha, \beta, \PtInf, \gamma)$ be a standardized virtual Heegaard diagram.
When we perform the connected sum between a curve $\beta_j$ and the boundary of a tubular neighborhood of $\gamma_i \cup \alpha_i$, $i,j \in \{1, ..., {\genre}\}$, we say that we perform a \emph{$G_{\gamma}$ move}. 
\end{defi}

\textit{Examples :}

The previously mentioned handle-slides of $\beta$ over $\alpha$, and $G_{\gamma}$ moves can be represented on the plane $\mathbb{R}^2$ by standard projections, as pictured in Figures \ref{fig:MouvGlissBA} and \ref{fig:MouvDoublGliss}, respectively. 

\begin{figure}[ht]
\centering
\includegraphics[scale=0.5]{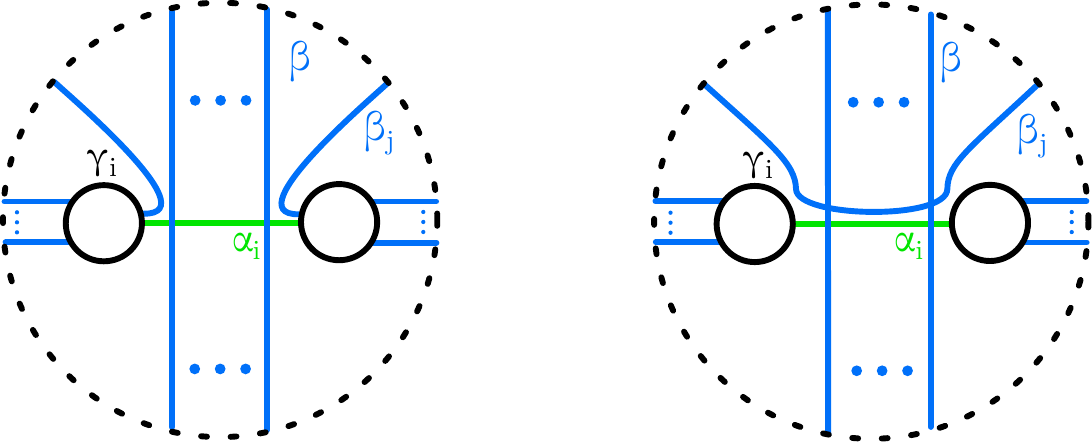}
\caption{Handle-slide from a curve $\beta_j$ over a curve $\alpha_i$ }\label{fig:MouvGlissBA}
\end{figure}

\begin{figure}[ht]
\centering
\includegraphics[scale=0.5]{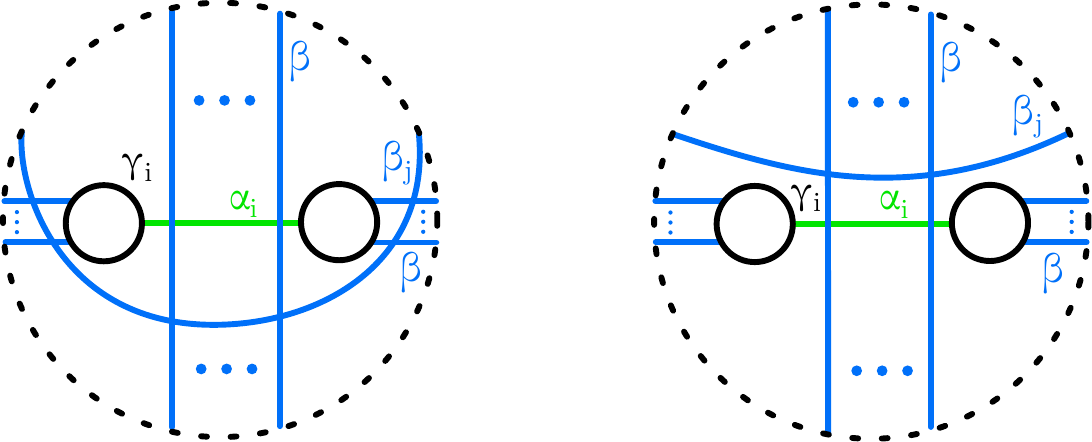}
\caption{Two diagrams differing by a $G_{\gamma}$ move from the curve $\beta_j$ over $\gamma_i \cup \alpha_i$}\label{fig:MouvDoublGliss}
\end{figure}

\begin{lemme}\label{lemme:ProjStandEqu}
Let us consider two standard projections of the same diagram $(\Sigma_{\genre}, \alpha, \gamma)$ along with a point $\PtInf \in \Sigma_{\genre} \setminus (\alpha \cup \gamma)$. 
These projections differ by a diffeomorphism of $\mathbb{R}^2 \setminus \big( \bigcup_{i \in \{1, ..., {\genre}\}} \mathring{D^1_i} \cup \mathring{D^2_i} \big)$ composed of isotopies fixing $\alpha \cup \gamma$, of Dehn twists along the boundary of disks, each containing a single pair $(\gamma_i, \alpha_i)_{ i \in \{1, ..., {\genre}\}}$, and of diffeomorphisms representing a generator of the braid group on ${\genre}$ strands. 
\end{lemme}

\begin{proof}
Let us consider two standard projections of the same diagram $D =(\Sigma_{\genre}, \alpha, \allowbreak \gamma)$ along with a point $\PtInf \in \Sigma_{\genre} \setminus (\alpha \cup \gamma)$. 
The images of $D$ by these projections differ by a diffeomorphism of $\mathbb{R}^2$ fixing the set of pairs $\{(\gamma_i, \alpha_i)\}_{i \in \{1, ..., {\genre}\}}$. 
Each pair $(\gamma_i, \alpha_i)_{ i \in \{1, ..., {\genre}\}}$ lies within a disk, as each $\gamma_i \cup \alpha_i$ admits a neighborhood isomorphic to a disk in $\mathbb{R}^2$.
Therefore, we can see the classes of these diffeomorphisms (modulo isotopies fixing $\alpha \cup \gamma$) as the set of braids on ${\genre}$ strands in the plane, with framing on each strand.

Artin's theorem gives a family of generators $(\omega_i)_{i \in \{1, ..., {\genre}-1\}}$ for these diffeomorphisms. 

\begin{theorem}[Artin \cite{Artin}]\label{th:Artin}
The braid group on $n$ strands in the plane admits a presentation with generators $\omega_1, ..., \omega_{n-1}$ and defining relations 
\begin{itemize}
\item $\omega_i \omega_j = \omega_j \omega_i$ if $|i-j| \geq 2$, $i,j \in \{1, ..., n-1\}$, 
\item $\omega_i \omega_{i+1} \omega_i = \omega_{i+1} \omega_i \omega_{i+1}$ for all $i \in \{1, ..., n-1\}$. 
\end{itemize}
\end{theorem}

More details in \cite[Th. 1.8]{Birman}.

Now, it remains to observe the difference between two diffeomorphisms for which their associated braids have differing framings.
Since the boundaries of the disks are fixed, the change induced on the framing by any considered diffeomorphism will be an element of $\mathbb{Z}^{\genre}$. 
A difference between two projections of one positive turn on the framing over one of the fixed disks amounts to performing, in the diagram, a Dehn twist around that disk. 
Therefore, we find the desired result. 
\end{proof}

We introduce virtual Heegaard diagrams because, unlike regular Heegaard diagrams, every virtual Heegaard diagram can be associated to a diagram whose family of $\beta$ curves can be entirely projected onto the plane $\mathbb{R}^2 \setminus \big( \bigcup_{i \in \{1, ..., {\genre}\}} D^1_i \cup D^2_i \big)$. 
We will call \emph{flat} such a diagram. 

\begin{defi}[Flat diagrams]
Let us consider a diagram $D=(\Sigma_{\genre}, \alpha, \beta) \in D^v_{He}$ and a fixed point $\PtInf \in \Sigma_{\genre} \setminus (\alpha \cup \beta)$. 
We say that $(D, \PtInf)$ is \emph{flat} with regard to a set of curves $\gamma \subset \Sigma_{\genre} \setminus \PtInf$  (or $\gamma$-flat) if the sets $\beta$ and $\gamma$ never intersect each other.

Let us consider two virtual Heegaard diagrams $D=(\Sigma_{\genre}, \alpha, \beta) \in D^v_{He}$ and $D'=(\Sigma_{\genre}, \alpha, \beta') \in D^v_{He}$, a fixed point $\PtInf \in \Sigma_{\genre} \setminus (\alpha \cup \beta \cup \beta')$ and a set of curves $\gamma \subset (\Sigma_{\genre} \setminus \{\PtInf\})$ such that the $\beta \cap \gamma$ intersections are transverse. 
We say that $(D', \PtInf)$ is a \emph{flattening} of $(D, \PtInf)$ with regards to $\gamma$ (or \emph{$\gamma$-flattening} of $(D, \PtInf)$) if $\beta'$ and $\gamma$ never intersect each other, and if $D'$ is linked to $D$ through a sequence of handle-slides of $\beta$ curves over $\alpha$, moves $R_{II}$ between $\beta$ curves or between $\beta$ and $\alpha$ curves, moves $R_{III}$ between $\beta$ curves and isotopies of the pointed surface $(\Sigma_{\genre}, \PtInf)$ with each move operated within an open set not containing $\PtInf$. 
We will sometimes simply say that $D'$ is a $\gamma$-flattening of $D$. 
\end{defi}

\Rq Two flattenings of the same virtual Heegaard diagram, with the same point $\PtInf$ and the same family of curves $\gamma$ will always be linked by a sequence of moves $R_{II}$ between $\beta$ curves or between $\beta$ and $\alpha$ curves, moves $R_{III}$ between $\beta$ curves, isotopies of the pointed surface $(\Sigma_{\genre}, \PtInf)$ and moves $G_{\gamma}$ (see Figure \ref{fig:MouvDoublGliss}) with each move operated within an open set not containing $\PtInf$. 
\vspace{5pt}

\begin{propo}\label{Aplatissement}
Let us consider a virtual Heegaard diagram $D=(\Sigma_{\genre}, \alpha, \beta) \in D^v_{He}$, and a fixed point $\PtInf \in \Sigma_{\genre} \setminus (\alpha \cup \beta)$.
Let $\gamma \subset (\Sigma_{\genre} \setminus \{\PtInf\})$ be a family of ${\genre}$ simple disjoint closed curves such that $\MatInter{\alpha}{\gamma} = \Id_{\genre}$, and that all intersections $\beta \cap \gamma$ are transverse. 
Then, there exists a $\gamma$-flattening of $D$. 
\end{propo}

\Rq We note that $\MatInter{\alpha}{\gamma} = \Id_{\genre}$ induces that $\alpha$ and $\gamma$ each have connected complement. 
\vspace{5pt}

\begin{proof}
Let us consider a virtual Heegaard diagram $D=(\Sigma_{\genre}, \alpha, \beta) \in D^v_{He}$, and a fixed point $\PtInf \in \Sigma_{\genre} \setminus (\alpha \cup \beta)$. 
Let $\gamma \subset (\Sigma_{\genre} \setminus \{\PtInf\})$ be a family of ${\genre}$ simple disjoint closed curves such that $\MatInter{\alpha}{\gamma} = (\delta_{ij})_{i,j \in \{1, ..., {\genre} \}}$, and that all intersections $\beta \cap \gamma$ are transverse. 
For all $i \in \{1, ..., \genre\}$, we denote by $p_i \in \gamma_i$ the intersection point defined by $\alpha_i \cup \gamma_i = \{p_i\}$.

Let us assume that the set $\beta \cap \gamma$ is non-empty (else the diagram is already flat). 
Then, there exists a crossing $c \in \beta \cap \gamma$ on a curve $\gamma_i \in \gamma$ such that the segment of $\gamma_i$ between $p_i$ and $c$ does not intersect $\beta$.
We denote by $\beta_j \in \beta$ the $\beta$ curve on which $c$ is placed.  

By an isotopy of the surface $\Sigma_{\genre} \setminus (\PtInf \cup \alpha)$, we can slide the crossing $c$ along the previously mentioned segment of $\gamma_i$ until it is sufficiently close to $p_i$. 
We can then perform a handle-slide of $\beta_j$ over $\alpha_i$.
Since an open segment of $\beta_j$ centered at $c$ is now parallel to an open segment of $\alpha_i$ centered at $p_i$, the crossing is erased by the handle-slide and no new crossing appeared because, by our initial hypothesis, $\alpha_i$ intersects $\gamma$ only at $p_i$. 

By induction on the number of crossings between $\beta$ and $\gamma$, we can recover a $\gamma$-flattening of $(D, \PtInf)$. 

\end{proof}

\Rq Since the $\alpha$ curves are all disjoint and $\gamma_i \cap \alpha = \{p_i\}$, we notice that the process used in the proof does not induce any additional crossing between $\beta$ and $\alpha$.
\vspace{5pt}

\begin{defi}
A standardized virtual Heegaard diagram $(\Sigma_{\genre}, \alpha, \beta, \PtInf, \gamma)$ is called \emph{standardized flat} if $((\Sigma_{\genre}, \alpha, \beta), \PtInf)$ is flat with regards to $\gamma$.
\end{defi}

It results from Proposition \ref{Aplatissement} and the Remark following Definition \ref{def:DiagStandardisé} that, for any virtual Heegaard diagram $(\Sigma_{\genre}, \alpha, \beta)$ and any point $\PtInf \in \Sigma_{\genre} \setminus (\alpha \cup \beta)$, there exists a family of curves $\gamma$ and a family of curves $\beta'$ such that $((\Sigma_{\genre}, \alpha, \beta'), \PtInf)$ is a $\gamma$-flattening of $((\Sigma_{\genre}, \alpha, \beta), \PtInf)$. 
Therefore, there always exists at least one standardized flat virtual Heegaard diagram $(\Sigma_{\genre}, \alpha, \beta', \PtInf, \gamma)$ associated to $(\Sigma_{\genre}, \alpha, \beta)$.

\subsection{The map $F''$ over virtual Heegaard diagrams}\label{Part:F''}

Let $D = (\Sigma_{\genre}, \alpha, \beta, \PtInf, \gamma)$ be a standardized flat virtual Heegaard diagram. 
Any standard projection $\psi$ of $((\Sigma_{\genre}, \alpha, \gamma), \PtInf)$ produces an image $(\alpha', \beta') = \psi(\alpha, \beta)$ in the plane with $2{\genre}$ disks removed. 
We call $\mathfrak{D}$ the set of all such images of standard projection. 

\begin{defi}
We define the map $F''_H : \mathfrak{D} \to \korps$ by the following algorithm. 
\end{defi}

Let us consider $(\alpha', \beta') \in \mathfrak{D}$. 
By definition of $\mathfrak{D}$, $(\alpha', \beta')$ is the image, by a standard projection, of a standardized flat virtual Heegaard diagram $D=(\Sigma_{\genre}, \alpha, \beta, \PtInf, \gamma)$. 
In the rest of this article, we will identify $\alpha'$ and $\beta'$ with their preimages under $\psi$, we will thus denote them $\alpha$ and $\beta$ respectively.

\begin{itemize}

\item We place a unique base point $p_i$ on each component $\beta_i$ of $\beta$. 

\item For every component $\alpha_k \in \alpha$ such that $\alpha_k \cap \beta \not= \emptyset$, we apply the action of $\Lambda$ to all crossings with $\beta$. 
In other words, let $n$ be the cardinal of $\alpha_k \cap \beta$, and $\Lambda_{(j)}$ be the $j$-nth factor of the coproduct $\Delta^{(n)}(\Lambda)$ (as indicated by \eqref{Not:Sweedler}).
We locally add to the $j$-nth crossing of $\alpha_k$, on the curve $\beta_i$ of $\beta$ involved in the crossing, the bead
 \begin{equation*}
      S^{d_j}(\Lambda_{(j)}) \hspace{5pt} \text{ where } d_j =
     \begin{cases}
     1 & \text{ if } (\vec{\alpha_k},\vec{\beta_i}) \text{ has negative orientation at the crossing} ,\\
     0 & \text{else (see Figure \ref{fig:PerlesKAmod}).}
     \end{cases}
\end{equation*} 
\begin{figure}[ht]
\centering
\includegraphics[scale=0.5]{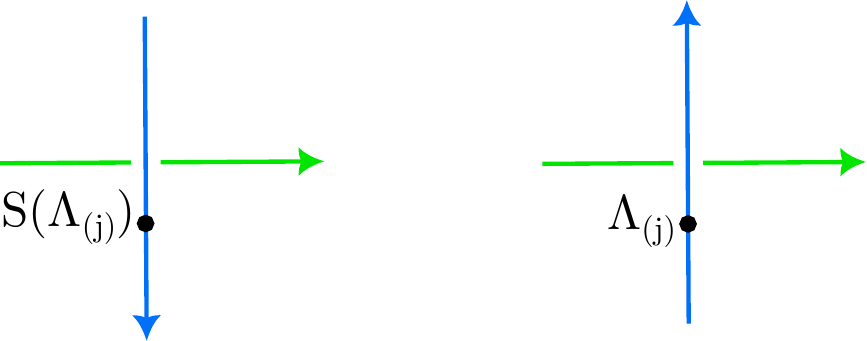}
\caption{Beads placement at the crossing between curves $\alpha_k$ and $\beta_j$}\label{fig:PerlesKAmod}
\end{figure}

We then work with sums of diagrams, where only the bead's values change (see in the example Figure \ref{fig:ExemplePerles}).
Accordingly, "the bead $x$" refers to the corresponding $x$ bead on every diagram in the summation. 

\item We then add on each "cap" and "cup" of $\beta$ the beads $g$ as indicated in Figure \ref{fig:PlaceG}. 
(We note that these are the same bead placements as in \ref{fig:FoncPerleGPP}.)

\begin{figure}[ht]
\centering
\includegraphics[scale=0.7]{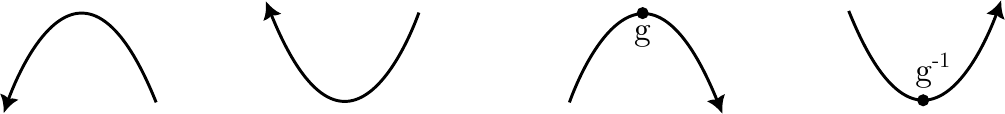}
\caption{Placement of $g$ beads on $\beta$}\label{fig:PlaceG}
\end{figure}

\item For each component $\beta_i \in \beta$, we follow its orientation starting from the base point, and multiply every bead accordingly (the multiplication is performed from the left).
This gives us a total bead $\perleTot_i$. 
If there is no bead on $\beta_i$, we consider $\perleTot_i = 1_H$. 

\item The image of $F''$ is given by 
$$F''_H(\psi(\alpha, \beta)) :=  \epsilon(\Lambda)^v \prod_{i =1}^{\genre} \mu(\perleTot_i), $$ 
where $v \in \mathbb{N}$ is the number of components $\alpha_k \in \alpha$ such that $\alpha_k \cap \beta = \emptyset$. 
\end{itemize}

\begin{figure}[ht]
\centering
\includegraphics[scale=0.45]{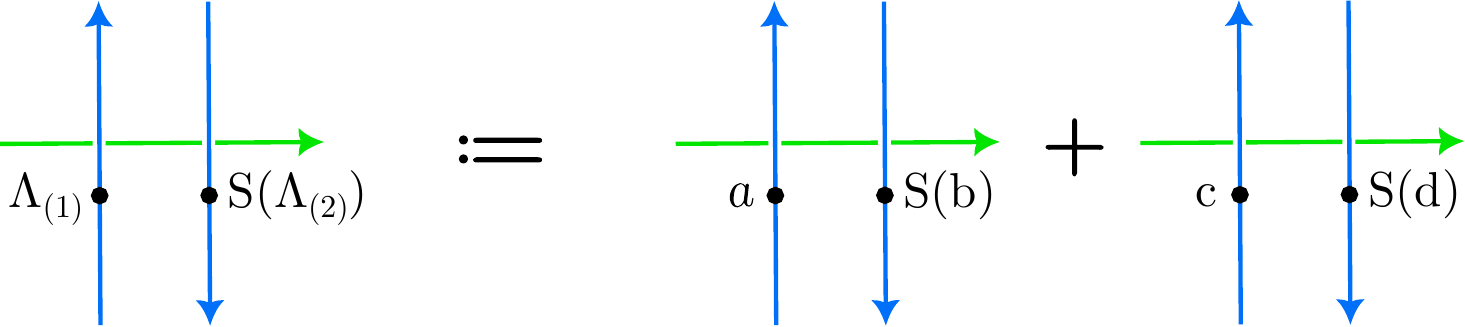}
\caption{Notation example when $\Delta^{(2)}( \Lambda )= a \otimes b + c \otimes d$}\label{fig:ExemplePerles}
\end{figure}

The map $F''$ is well defined because, by property of the symmetrized integral, for all $x, y \in H$ we have $\mu(xy) = \mu(yx)$. 
Hence, $F''$ is independent of the choice of base point placement.

\begin{propo}\label{ExistF''_H}
Let $(\Sigma_{\genre}, \alpha, \beta, \PtInf, \gamma)$ be a standardized flat virtual Heegaard diagram. 
Two elements of $\mathfrak{D}$ recovered from $(\Sigma_{\genre}, \alpha, \beta, \PtInf, \gamma)$ will have the same image by $F''$. 
Therefore, we can define 
$$ F''(\Sigma_{\genre}, \alpha, \beta, \PtInf, \gamma) = F''(\psi(\alpha, \beta)) $$
where $\psi$ is any standard projection of $((\Sigma_{\genre}, \alpha, \gamma), \PtInf)$. 
\end{propo}

\begin{proof}
Let $D=(\Sigma_{\genre}, \alpha, \beta, \PtInf, \gamma)$ be a standardized flat virtual Heegaard diagram. 

We have previously seen in Lemma \ref{lemme:ProjStandEqu} that two projections differ by a diffeomorphism composed of isotopies (fixing $\alpha$ and $\gamma$ curves), of Dehn twists around disks containing each a pair $(\gamma_i, \alpha_i)_{ i \in \{1, ..., {\genre}\}}$, and of diffeomorphisms representing the generators of the braid group on ${\genre}$ strands.
Thus, we want to show that $F''_H$ is invariant by these elements. 

Let us consider the images of two standard projections of the same diagram, and assume that these projections differ only by an isotopy of the plane fixing the $\alpha$ and $\gamma$ curves. 
Such an isotopy does not change the crossings between $\alpha$ and $\beta$ on the images of the diagram.  
If it modifies the extrema of the $\beta$ curves between the images of both projections, the change induced will always be the appearance (or vanishing) of a pair of a consecutive minimum and maximum with the same orientations. 
Hence, depending of the orientation of the curves affected by the isotopy, either no bead appears, or a bead colored by $g$ and a bead colored by $g^{-1}$ will appear consecutively. 
Therefore, $F''_H$ is invariant by isotopies of the plane fixing the $\alpha$ and $\gamma$ curves. 

The computational algorithm of $F''$ allows us to note the following point ;
If two elements of $\mathfrak{D}$ have the same crossings $\alpha \cap \beta$ (orders and orientations over the curves $\alpha$ and $\beta$) and the same local minima and maxima over the $\beta$ curves (positions on $\beta$ relative to the crossings $\alpha \cap \beta$ and orientations included), then these two diagrams will have the same image by $F''_H$. 
In the case of images by two standard projections of the same diagram, they directly admit the same crossings $\alpha \cap \beta$, therefore, it is enough to focus on the minima and maxima of $\beta$ curves. 

We can also deduce, in a more general setting, that two elements of $\mathfrak{D}$ differing by the image of an $R_{II}$ move (see Figure \ref{fig:MouvR2}), an $R_{III}$ move (see Figure \ref{fig:MouvR_3}) or a $G_{\gamma}$ move (see Figure \ref{fig:MouvDoublGliss}) on $\Sigma \setminus (\alpha \cup \gamma \cup \PtInf)$ will have the same image by $F''_H$, for these moves do not influence the crossings between $\beta$ and $\alpha$, neither do they influence the extrema of $\beta$. 
In the rest of this article, we sometimes use these moves to simplify diagrams without changing their image by $F''_H$. 

\begin{figure}[ht]
\centering
\includegraphics[scale=0.65]{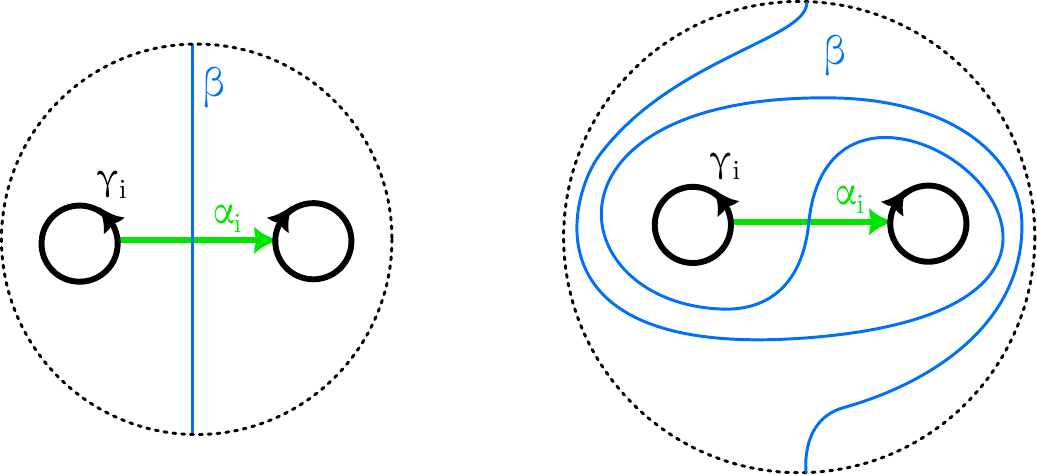}
\caption{Two projections of the same diagram differing by a Dehn twist}\label{fig:TwistDehn}
\end{figure}

Let us consider two projections of the same diagram that differ only by a Dehn twist around a fixed disk containing a pair $(\gamma_i, \alpha_i)$ (and not containing any other pair $(\gamma_j, \alpha_j), j \not = i$).
We can see this difference in Figure \ref{fig:TwistDehn} (where the vertical segment $\beta$ represents multiple parallel strands of $\beta$). 
For each strand $\beta_k \in \beta$ intersecting $\alpha_i$ in the right figure (Figure \ref{fig:TwistDehn}) we collect a bead $b'_i$ on the represented segment. 
If the segment is oriented upward, we see a bead $g^{-1}$ appear before the crossing $\alpha_i \cap \beta_k$ and a bead $g$ after the same crossing. 
The same situation arises if the segment is oriented downward. 
Therefore, $b'_i = g S^{d_k}(\Lambda_{(i_k)}) g^{-1} = S^{d_k +2}(\Lambda_{(i_k)})$, with $d_k = 0$ or $1$ depending on the sign on the crossing $\alpha_i \cap \beta_k$, and $i_k$ is the index indicating the order in which $\alpha_i$ crosses the strands of $\beta$. 
However, we recall that $\bigotimes_j S^2(\Lambda_{(j)}) = \bigotimes_j (S^2(\Lambda))_{(j)} = \bigotimes_j \Lambda_{(j)}$ because $H$ is unimodular. 
Hence, on each segment we recover the bead collected on the corresponding segment in the left figure (Figure \ref{fig:TwistDehn}). 
Thus, both projections will give the same image by $F''$. 

The generators $\omega_1, ..., \omega_{{\genre}-1}$ of the braid group on ${\genre}$ strands, mentioned in Theorem \ref{th:Artin}, can be represented by the following diffeomorphisms (see Figure \ref{fig:GenerateurTresses}). 
Let $W_1, \allowbreak ..., W_{\genre}$ be closed connected neighborhoods in $ \mathbb{R}^2$ of the pairs $(\alpha_1 \cup \gamma_1), ..., (\alpha_{\genre} \cup \gamma_{\genre})$, respectively. 
For all $i,j \in \{1, ..., {\genre}\}$, we consider $V_i$ a ring ($\simeq \mathbb{S}^1 \times ]-\epsilon, \epsilon[$ with $\epsilon >0$) such that $(W_i \cup W_{i+1}) \subset V_i$ and $V_i \cap W_j = \emptyset$ for all $j \not \in \{i, i+1\}$.
Each diffeomorphism representing a generator $\omega_i$ is isotopic to the identity on $\operatorname{Diff}(\mathbb{R}^2)$ by an isotopy of $\mathbb{R}^2$. 
This isotopy is the identity on $\mathbb{R}^2  \setminus V_i$ and rotates $W_i$ and $W_{i+1}$ inside $V_i$ by an angle $t \in [0, \pi]$, while keeping the $\alpha$ curves horizontal. 
(Each diffeomorphism is isotopic to the identity inside the diffeomorphisms of $\mathbb{R}^2$, but not inside those of $ \mathbb{R}^2  \setminus \underset{i \in \{1, ..., {\genre}\}, j \in \{1,2\}}{\bigcup} \mathring{D_i^j} $.)

The action of a generator $\omega_i$ on a standard projection translates graphically as shown in Figure \ref{fig:GenerateurTresses} (where the vertical segment $\beta$ represents multiple parallel strands of $\beta$). 
\begin{figure}[ht]
\centering
\includegraphics[scale=0.68]{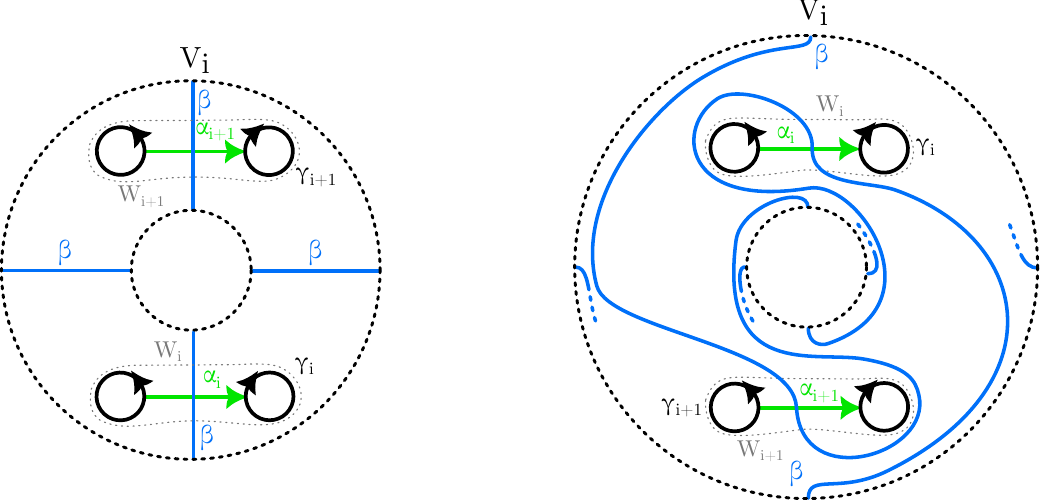}
\caption{effect of $\omega_i$ on a diagram's projection $(\Sigma, \alpha, \beta)$ (initial diagram on the left, effect of $\omega_i$ on the right)}\label{fig:GenerateurTresses}
\end{figure}
If a strand of $\beta$ (other than those represented in Figure \ref{fig:GenerateurTresses}) intersects the neighborhood $V_i$, we can make $V_i$ "sufficiently thin" for that strand to simply cut across $V_i$. 
It will be subject to an isotopy on $V_i$, which, as seen, previously, does not influence its image by $F''_H$. 

If a strand of $\beta_k$ intersecting $\alpha_i$ or $\alpha_{i+1}$ in Figure \ref{fig:GenerateurTresses} is oriented upward, no bead will appear on its extrema. 
If the strand of $\beta$ intersecting $\alpha_i$ (respectively $\alpha_{i+1}$) is oriented downward, before its intersection with $\alpha_i$ (resp. after its intersection with $\alpha_{i+1}$) a bead $g^{-1}$, and then a bead $g$, will appear consecutively on the right figure (Figure \ref{fig:GenerateurTresses}). 
During the bead multiplication, it will result in a trivial bead $1_H$. 
Hence, both projections will give the same image by $F''_H$. 

Therefore, $F''_H$ is independent of the choice of the standard projection used to send the diagram onto $\mathbb{R}^2$ with $2{\genre}$ disks removed. 

\end{proof}

\begin{propo}\label{prop:ChangOrientation}
The map $F''$ is invariant under orientation reversal of a $\beta$ curve as well as invariant under orientation reversal of an $\alpha$ curve and its associated $\gamma$ curve. 
\end{propo}

\begin{proof}

Let $D=(\Sigma_{\genre}, \alpha, \beta, \PtInf, \gamma)$ be a flat virtual Heegaard diagram.
Let us consider a virtual Heegaard diagram $(\Sigma_{\genre}, \alpha, \overline{\beta})$ differing from $(\Sigma_{\genre}, \alpha, \beta)$ only over the orientation of a curve $\beta_i$ (we will denote $\overline{\beta_i} \in \overline{\beta}$ the curve $\beta_i$ with reverse orientation). 
The diagram $\overline{D}=(\Sigma_{\genre}, \alpha, \overline{\beta}, \PtInf, \gamma)$ obtained this way is also a flat virtual Heegaard diagram.
We project $D$ and $\overline{D}$ onto the plane with the same standard projection $\psi$.  

We start by decomposing $\beta_i$ into segments over which we will observe the effects of the orientation reversal.
We divide $\beta_i \in \beta$ into two types of segments, 
\begin{enumerate}
\item a first type being segments containing only a local maximum, then a number $n \in \mathbb{N}$ of crossings between $\beta_i$ and $\alpha$ (potentially with crossing between $\beta_i$ and $\beta$ in between), and then a local minimum,
\item a second type being segments forming the complement over $\beta_i$ of first-type segments. 
\end{enumerate}

During the computation of $F''$, the symmetrized integral's property $\mu(zy) = \mu(yz)$ allows us to assume that the base points of $\beta_i$ and $\overline{\beta_i}$ are at the end of a segment.
Therefore, during the bead collection, we can consider one bead per segment.
On a segment, we denote by $x \in H$ the collected bead on $\beta_i$ and by $\overline{x} \in H$ the collected bead on $\overline{\beta_i}$. 

By definition, $\beta_i$ is oriented upward on second-type segments, hence, only beads originating from positive crossings of $\beta_i$ with $\alpha$ will appear on these segments. 
Thus, the definition of beads on crossings, together with the relation $S(x) S(y) = S(yx)$, allows us to see that $\overline{x} = S(x)$. 

Over first-type segments, we notice that every crossing of $\beta_i$ with $\alpha$ will be negative. 
The relation $\overline{x} = S(x)$ does not directly appear but can be obtained by properties of the antipode and pivot. 

\begin{figure}[ht]
\centering
\includegraphics[scale=0.45]{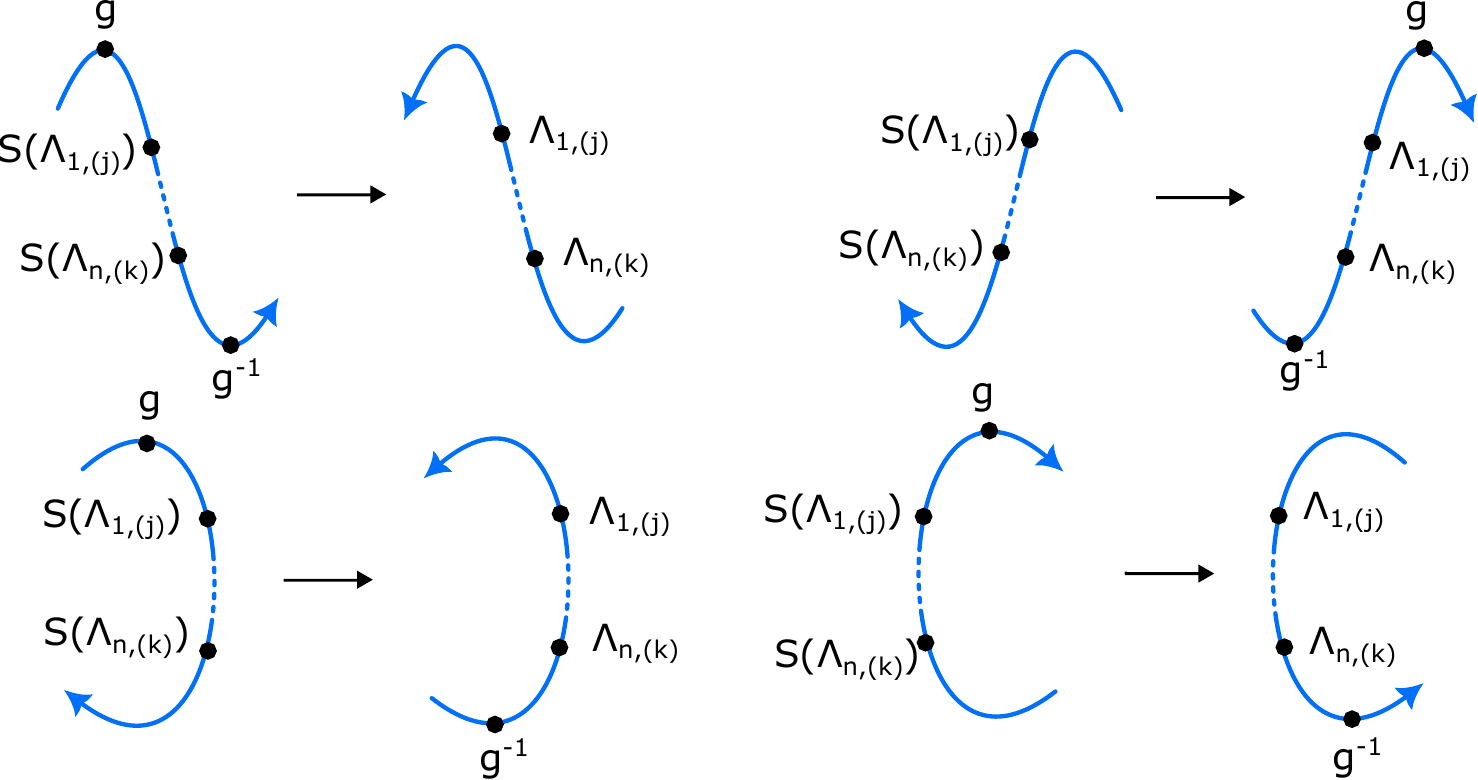}
\caption{Orientation reversal on first-type segments (with $n$ beads originating from crossings between $\beta_i$ and $\alpha$). }\label{fig:SegmentsBDesc}
\end{figure}

For example, in the case represented in the upper left side of Figure \ref{fig:SegmentsBDesc}, we notice that $x = g^{-1}S(\Lambda_{n,(k)}) \cdots S(\Lambda_{1,(j)})g$ and $\overline{x} = \Lambda_{1,(j)} \cdots \Lambda_{n,(k)}$ where each $\Lambda_{p,(q)}$, $p, q \in \mathbb{N}$, originates from a $\Lambda_p= \Lambda$ by coproduct, as indicated in notation \eqref{Not:Sweedler}. 
However, 
$$ S(g^{-1}S(\Lambda_{n,(k)}) \cdots S(\Lambda_{1,(j)})g) = g^{-1}S^2(\Lambda_{1,(j)}) \cdots S^2(\Lambda_{n,(k)})g = \Lambda_{1,(j)} \cdots \Lambda_{n,(k)}. $$ 
We then recover the relation $\overline{x} = S(x)$ for that case.
All other cases are represented in Figure \ref{fig:SegmentsBDesc} and result in the same relation through analogous reasoning. 
Therefore, every bead originating from a first type segment verifies the relation $\overline{x} = S(x)$. 

Let $\perleTot_i$ and $\overline{\perleTot_i}$ be the total beads of $\beta_i$ and $\overline{\beta_i}$, respectively. 
Since $\mu \circ S = \mu$, we deduce that $\mu(\perleTot_i) = \mu (\overline{\perleTot_i})$, which allows us to conclude the invariance of $F''$ under reversal of a $\beta$ curve's orientation. 

\vspace{5pt}

Let us consider this time a diagram $\overline{D}$ differing from $D$ only over the orientation of a curve $\alpha_i$ (we will denote $\overline{\alpha_i} \in \overline{\alpha}$ the curve $\alpha_i$ with reverse orientation). 
The family of curves $\overline{\gamma}$ associated to $\overline{D}$ will be identical to $\gamma$ except on the curve $\gamma_i$ where it will have reverse orientation (we will denote $\overline{\gamma_i} \in \overline{\gamma}$ the curve $\gamma_i$ with reverse orientation). 
We can then consider standard projections such that the images of $D$ and $\overline{D}$ differ only on an open disk such as represented in Figure \ref{fig:DemiTwistDehn}.

\begin{figure}[ht]
\centering
\includegraphics[scale=0.7]{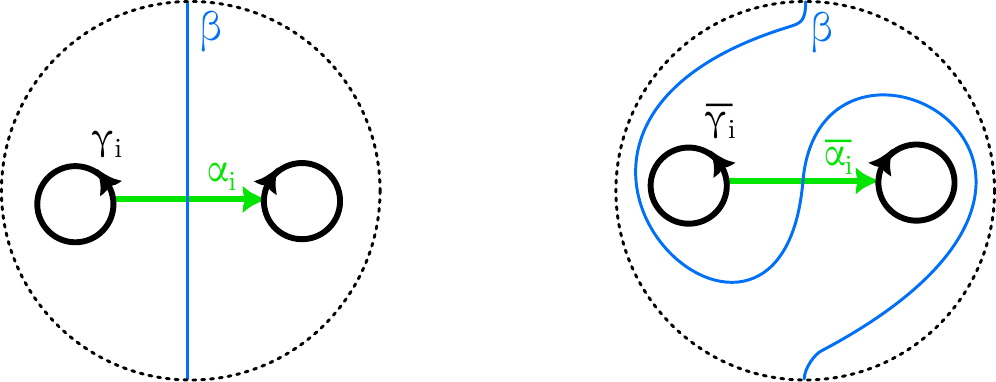}
\caption{Impact on the image in the plane of an orientation reversal on a curve $\alpha_i$ (here $\beta$ represent multiple parallel $\beta$ curves) }\label{fig:DemiTwistDehn}
\end{figure}

Here, we will use the property $\Delta^{(n)} (S(x)) = S(x_{(n)}) \otimes \cdots \otimes S(x_{(1)})$ from the coproduct and antipode. 
If a strand from $\beta$ is oriented upward, then no bead $g^{\pm 1}$ appears after the change.
However the bead $\Lambda_{(i)}$ at the crossing becomes $S(\Lambda_{(l-i+1)}) = S(\Lambda)_{(i)}$, with $l \in \mathbb{N}$ the number of crossings of $\alpha_i$ with $\beta$. 
If a strand of $\beta$ is oriented downward, then a bead $g$ and a bead $g^{-1}$ appear before and after the crossing, respectively. 
Furthermore, the bead $S(\Lambda_{(i)})$ at the crossing becomes $\Lambda_{(l-i+1)} = S^{-1}(S(\Lambda)_{(i)})$.
The product of these three beads is $g\Lambda_{(l-i+1)}g^{-1}=S(S(\Lambda)_{(i)})$ as $g$ is a pivot of $H$.
Since $H$ is unimodular we have $S(\Lambda) = \Lambda$, which induces that, for all $n \in \mathbb{N}$, $\Lambda_{(1)} \otimes \Lambda_{(2)} \otimes \cdots \otimes \Lambda_{(n)} = S(\Lambda_{(n)}) \otimes \cdots \otimes S(\Lambda_{(2)}) \otimes S(\Lambda_{(1)})$. 
Therefore, both projections will give the same image by $F''$, thus $F''_H$ is invariant under orientation reversal of an $\alpha$ curve and its associated $\gamma$ curve. 

\end{proof}
The map $F''_H$ can be extended to all virtual Heegaard diagrams as follows. 

\begin{propo}
Let $D=(\Sigma_{\genre}, \alpha, \beta, \PtInf, \gamma)$ be a standardized virtual Heegaard diagram. 
We define the image of $D$ under $F''_H$ as follows
$$ F''_H(D) := F''_H(\Sigma_{\genre}, \alpha, \beta', \gamma,  \PtInf ) $$
where $((\Sigma_{\genre}, \alpha, \beta'),\PtInf)$ is a $\gamma$-flattening of $((\Sigma_{\genre}, \alpha, \beta),\PtInf)$. 
The image $ F''_H(D)$ is independent of the choice of $\gamma$-flattening. 
\end{propo}

\begin{proof}
Let $D=(\Sigma_{\genre}, \alpha, \beta, \PtInf, \gamma)$ be a standardized virtual Heegaard diagram. 

For the image of $F''_H(D)$ to be well defined, it is necessary that $F''_H$ has the same image for any $\gamma$-flattening of $((\Sigma_{\genre}, \alpha, \beta), \PtInf)$. 
We have mentioned that, by definition, two ($\gamma$-)flattenings of the same virtual Heegaard diagram are linked by a sequence of $R_{II}$ moves between $\beta$ curves or between $\beta$ and $\alpha$ curves, $R_{III}$ moves between $\beta$ curves, isotopies of the pointed surface $(\Sigma_{\genre},\PtInf)$, and $G_{\gamma}$ moves, with each move operated within an open set not containing $\PtInf$. 
Hence, we want to prove that $F''_H$ is invariant by these operations. 
We also recall (seen in the proof of Proposition \ref{ExistF''_H}) that only operations which alter the crossings between $\alpha$ and $\beta$ curves, or the $\beta$ curves' extrema, can change the image under $F''_H$. 
\vspace{5pt}

Two standardized virtual Heegaard diagrams differing by an isotopy of the pointed surface $(\Sigma_{\genre}, \PtInf)$ admit two standard projections differing by an isotopy of the plane leaving $\alpha \cup \gamma$ fixed.
However, we have also shown (in proof of Proposition \ref{ExistF''_H}) that $F''_H$ is invariant by isotopies of the plane leaving $\alpha \cup \gamma$ fixed, by $G_{\gamma}$ moves, and by $R_{III}$ and $R_{II}$ moves between $\beta$ curves in $\Sigma \setminus (\alpha \cup \gamma \cup \PtInf)$. 
Therefore, two flattenings differing by these moves will have the same image. 

We now have only to show that $F''$ is invariant by $R_{II}$ moves between $\beta$ and $\alpha$ curves. 

\vspace{5pt}
Let $D=(\Sigma_{\genre}, \alpha, \beta, \PtInf, \gamma)$ and $D'= (\Sigma_{\genre}, \alpha, \beta', \allowbreak\PtInf, \gamma)$ be two standardized virtual Heegaard diagrams, identical except on an open set $V$ where they differ by an $R_{II}$ move as shown in Figure \ref{fig:MouvR_2alpha}.

\begin{figure}[ht]
\centering
\includegraphics[scale=0.5]{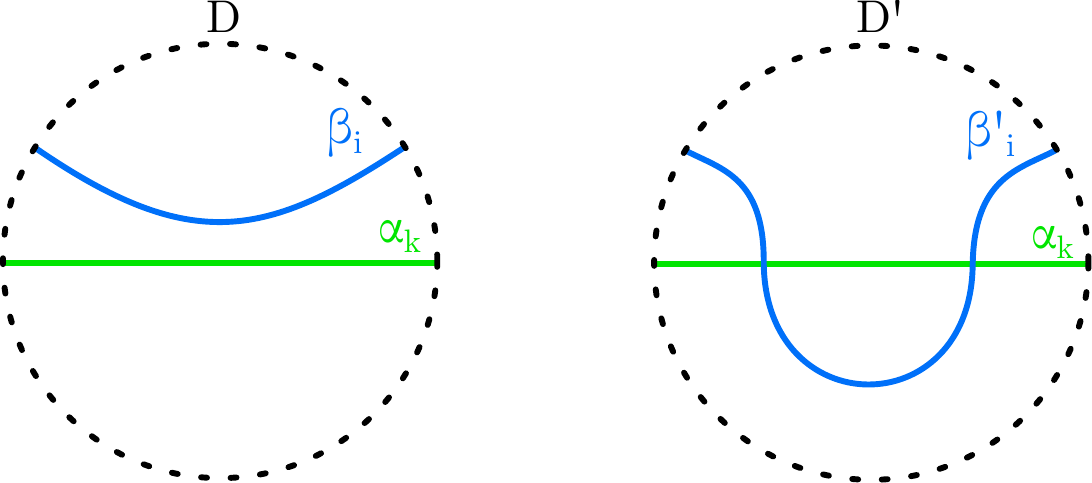}
\caption{$R_{II}$ move between an $\alpha$ curve and a $\beta$ curve}\label{fig:MouvR_2alpha}
\end{figure}

Let us consider a component $\beta_i \in \beta$, $i \in \{1, ..., \genre\}$, we denote by $\perleTot_i$ the total bead recovered on the left diagram (Figure \ref{fig:MouvR_2alpha}) and $\perleTot'_i$ the total bead recovered on the right diagram.
If the curve $\beta_i$ is oriented from right to left, then no bead $g$ appears on the diagrams. 
The curve $\alpha_k \in \alpha$, $k \in \{1, ..., \genre\}$, involved will always be oriented from left to right, by definition of a standard projection. 
We denote by $m \in \mathbb{N}$ the integer such that the $R_{II}$ move between $\alpha_k$ and $\beta_i$ is performed just after $m-1$ crossings. 
Hence, let $A_{i,m}$ and $B_{i,m}$ be the products of collected beads on $\beta_i$ before and after the considered open set, respectively, we have, for the left diagram, $\perleTot_i = B_{i,m} A_{i,m}$ and, for the right diagram, $\perleTot_i' = B_{i,m} (\Lambda_{k,(m)})_{(1)} S((\Lambda_{k,(m)})_{(2)}) A_{i,m}$. 

By definition of the antipode, $(\Lambda_{k,(m)})_{(1)} S((\Lambda_{k,(m)})_{(2)}) = \epsilon(\Lambda_{k,(m)})$ and, by property of the coproduct, for all $n \geq 1$ and all $x \in H$, we have $x_{(1)} \otimes \cdots \otimes x_{(m-1)} \otimes \epsilon(x_{(m)}) \otimes x_{(m+1)} \otimes \cdots \otimes x_{(n)} = \Delta^{(n-1)}(x)$. 
Let $l \geq 2$ be the number of crossings between $\beta$ and $\alpha_k$. 
If $l > 2$, the previous equality allows us to recover the action of $\Lambda_k$ on the remaining $l-2$ crossings with $\alpha_k$. 
If $l = 2$, the equality produces the term $\epsilon(\Lambda_k)=\epsilon(\Lambda)$ which is the term added to the final product when a curve $\alpha_k$ does not intersect $\beta$. 
Thus we find the same images by $F''_H$ for the right diagram and for the left diagram (in Figure \ref{fig:MouvR_2alpha}).

\vspace{5pt}

If the curve $\beta_i$ is oriented from left to right, then a bead $g^{-1}$ appears on its local minimum. 
In the same way as before, we get for the left diagram the total bead $\perleTot_i = B_{i,m} g^{-1} A_{i,m}$ and for the right diagram the total bead $\perleTot_i' = B_{i,m} (\Lambda_{k,(m)})_{(2)} g^{-1} \allowbreak S((\Lambda_{k,(m)})_{(1)}) A_{i,m}$. 
By definition of the antipode $S$ and pivot $g$, we can simplify the bead $\perleTot_i'$. 
\begin{align*}
B_{i,m} (\Lambda_{k,(m)})_{(2)} g^{-1} S((\Lambda_{k,(m)})_{(1)}) A_{i,m} & = B_{i,m} g^{-1} S^2((\Lambda_{k,(m)})_{(2)})  S((\Lambda_{k,(m)})_{(1)}) A_{i,m}\\
& = B_{i,m} g^{-1} S((\Lambda_{k,(m)})_{(1)}  S((\Lambda_{k,(i+1)})_{(2)})) A_{i,m} \\
& = B_{i,m} g^{-1} \epsilon(\Lambda_{k,(m)}) A_{i,m}.
\end{align*}

Therefore, the same equalities and reasoning as before apply, and we recover the same final image under $F''_H$ for both diagrams. 

Analogous reasoning also applies for an $R_{II}$ move where the $\beta$ curve is placed under the $\alpha$ curve. 
Hence, we can conclude our proof. 

\end{proof}

\Rq It is strongly suspected, but remains unproven, that the map $F''_H$ on standardized virtual Heegaard diagrams is also independent of the choice of the family of curves $\gamma$.

\subsection{The chromatic spherical invariant $\mathcal{K}$}

In this part, we will prove the link between the $3$-manifolds invariant $\mathcal{K}_{\H-mod}$ defined in \cite{CGPT} (with $\H-mod$ the category of finite-dimensional pivotal $H$-modules)  and the map $F''_H$ previously defined.
Before proving this link, we start by defining $\mathcal{K}_{\H-mod}$ using results from \cite{CGPT}. 

We consider the left $H$-module $H$ defined by the action $L:H \to \operatorname{End}_{\korps}(H)$ giving $L_h(x) = hx$ for all $h,x \in H$. 
In particular, $H$ seen as such is a projective generator of $\H-mod$. 
We denote by $\Proj_H$ the ideal of projective modules from the pivotal category $\H-mod$. 

The definition of $\mathcal{K}_{\H-mod}$ relies on the category of $\H-mod$-colored ribbon graphs over multi-handlebodies as defined below (see \cite{TuraevQuantum}). 
We call \emph{multi-handlebodies} a disjoint union of a finite number of $3$-dimensional oriented handlebodies.

A \emph{$\H-mod$-colored ribbon graph} over a multi-handlebody $\mathfrak{C}$ is a finite graph embedded in $\partial \mathfrak{C}$ with each edge colored with an object of $\H-mod$ and each vertex thickened in $\partial \mathfrak{C}$ into a coupon colored by a morphism of $\H-mod$.  
Every coupon has a top and a bottom.
Since the graphs are considered on an oriented surface, they possess a natural framing with its ribbon structure given by thickening the graph in the surface $\partial \mathfrak{C}$ (see \cite{TuraevQuantum} for more details on ribbon graphs).
We say that a graph is \emph{$\Proj_H$-colored} if all its colors are elements of $\Proj_H$.

\begin{defi}
A \emph{bichrome graph} on a handlebody is a graph on the boundary of a handlebody that is partitioned into two disjoint subgraphs, one red, the other blue, such that the blue subgraph is $\H-mod$-colored and that the red subgraph is a disjoint union of unoriented simple closed curves. 
\end{defi}

For the computation of $\mathcal{K}_{\H-mod}$ which follows, we will consider $\Proj_H$-colored graphs and we can even restrict ourselves to graphs where all edges are colored by $H$.

On bichrome graphs, it is possible to apply a "red to blue" or "chromatic" modification (see Figure \ref{fig:AppChroma}). 
It is a local modification allowing a curve from the red subgraph to pass into the blue subgraph thanks to a category morphism, called chromatic morphism (more details in \allowbreak \cite{CGPT}).

\begin{figure}[ht]
\centering
\includegraphics[scale=0.7]{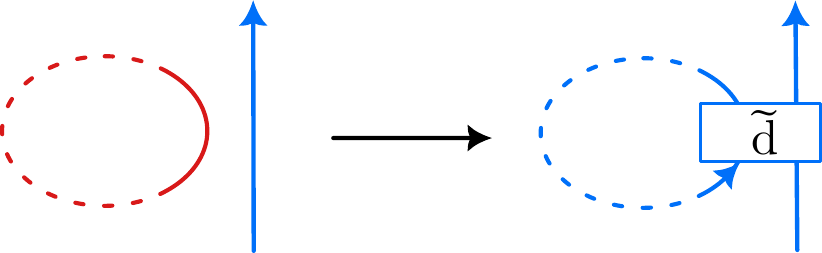}
\caption{Chromatic modification with the chromatic morphism $\tilde{d}$.}\label{fig:AppChroma}
\end{figure}

By Theorem \cite[2.5]{CGPT}, since $H$ is a spherical Hopf algebra, there exists a non-degenerate modified trace $\trace$ (defined in \cite{GPV}) over the ideal $\Proj_H$, and a chromatic morphism $\tilde{d} : H \otimes H  \to H \otimes H$. 
This trace is unique (up to a scalar) and determined by $\trace_H(f) = \mu(f(1))$ for all $f \in \operatorname{End}_H(H)$ (see \cite[Th.1]{Bel_Bla_Gai}).  
We define a chromatic morphism $\tilde{d}: H \otimes H \to H \otimes H$ (introduced in \cite{CGPT}) by 
$$ \tilde{d}(x \otimes y) = \lambda(S(y_{(1)} )gx) y_{(2)} \otimes y_{(3)} \qquad \text{ for } x \otimes y \in H \otimes H.$$

Let $\mathcal{H}_{B}$ be the set of pairs $(\mathfrak{C}, \Gamma)$ with $\mathfrak{C}$ a handlebody and $\Gamma$ a bichrome graph on $\mathfrak{C}$ with a non-empty blue subgraph.
Let $F$ be the pivotal functor associated with Penrose graphical calculus (more details in \cite{TuraevQuantum}).   
\cite[Theorem 2.1]{CGPT} applied to the pivotal category $\H-mod$ allows us to define $F'$, an invariant on $\mathcal{H}_B$, through the following theorem. 

\begin{theorem}\label{Th:DefF'}
There exists a unique map $F' : \mathcal{H}_B \to \korps$ satisfying the following four conditions.  
\begin{enumerate}
\item The element $F'(\mathfrak{C},\Gamma) \in \korps$ depends only on $(\mathfrak{C},\Gamma) \in \mathcal{H}_B$, up to orientation preserving diffeomorphisms. 
\item For all $\Proj_H$-colored ribbon graphs $(B^3, \Gamma)$ on the $3$-ball, we have \\
$F'(B^3,\Gamma) = \trace(F(T))$ where $T$ is a bichrome graph representing an element of $\operatorname{End}_{\korps}(P)$, with $P \in \Proj_H$, admitting $\Gamma$ as its braid closure.  
\item For all $(\mathfrak{C}_1, \Gamma_1), (\mathfrak{C}_2, \Gamma_2) \in \mathcal{H}_B$, we have 
$$ F'(\mathfrak{C}_1 \sqcup \mathfrak{C}_2, \Gamma_1 \sqcup \Gamma_2) = F'(\mathfrak{C}_1, \Gamma_1) F'(\mathfrak{C}_2, \Gamma_2). $$
\item Let $D$ be a $2$-disk disjoint from the red subgraph and intersecting at least one blue edge. 
If we cut $(\mathfrak{C}, \Gamma) \in \mathcal{H}_{\Proj_H}$ along $D$ (as described below), we have $F'(cut_D(\mathfrak{C}), cut_D(\Gamma)) = F'(\mathfrak{C},\Gamma)$. 
\item If $\Gamma, \Gamma' \in \mathcal{H}_B$ are linked by a chromatic modification, then $F'(\Gamma) = F'(\Gamma')$. 
\end{enumerate}
\end{theorem}

We now define what \emph{cutting} a multi-handlebody and a bichrome graph means for us. 
Cutting an element $(\mathfrak{C},\Gamma) \in \mathcal{H}_B$ along a disk $D$ (with $\partial D \subset \partial \mathfrak{C}$) gives a new multi-handlebody and a new graph that we will denote $cut_D(\mathfrak{C})$ and $cut_D(\Gamma)$, respectively. 
The graph $cut_D(\Gamma)$ is obtained from $\Gamma$ by replacing the intersection $\partial D \cap \Gamma$ by a pair of coupons colored by the morphisms $\sum_i x_i \otimes_{\korps} y_i$ with $\{x_i\}_i$ a basis of $\operatorname{Hom}_{\H-mod}(\mathbb{1}, P)$ and $\{y_i\}_i$ the dual basis of $\operatorname{Hom}_{\H-mod}(P, \mathbb{1})$ with regards to the pairing $(x,y) \in \operatorname{Hom}_{\H-mod}(\mathbb{1}, P) \times \operatorname{Hom}_{\H-mod}(P, \mathbb{1}) \mapsto \trace_P(xy)$, where $P \in \Proj_H$ is the projective element present on the intersection $\Gamma \cap \partial D$.  
The multi-handlebody $cut_D(\mathfrak{C})$ is obtained by cutting out a tubular open neighborhood of the disk $D$ from $\mathfrak{C}$, i.e. $cut_D(\mathfrak{C}) = \mathfrak{C} \setminus (V(D))$ with $V(D) \simeq D  \: \times \: ] - \epsilon, \epsilon[$, for $\epsilon > 0$. 
We can see this operation represented in Figure \ref{fig:Coupure}. 

\begin{figure}[ht]
\centering
\includegraphics[scale=1]{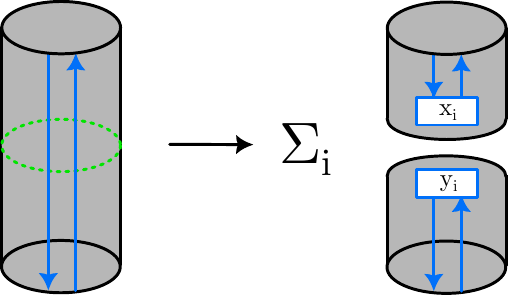}
\caption{A cutting operation}\label{fig:Coupure}
\end{figure}

Let $M$ be a $3$-dimensional connected oriented closed manifold and $M= \mathfrak{C}_\alpha \cup \mathfrak{C}_\beta$ be a Heegaard splitting of $M$. 
We denote by $(\Sigma_{\genre}, \alpha, \beta)$ the associated Heegaard diagram, where $\alpha$ and $\beta$ are sets of curves associated to the handlebodies $\mathfrak{C}_\alpha$ and $\mathfrak{C}_\beta$, respectively. 
Let us consider a ribbon graph $O_H \subset \Sigma \setminus (\alpha \cup \beta)$ defined as the braid closure on a disk of a coupon colored by $h \in \operatorname{End}_{\H-mod}(H)$ such that $F'(O_H) = 1$. 
We call $\Gamma$ the bichrome diagram on the surface $\Sigma_{\genre} = \partial\mathfrak{C}_\alpha = \partial\mathfrak{C}_\beta$ formed by the set of curves $\beta$ considered as red curves and the graph $O_H$ mentioned above seen as the blue subgraph. 

\begin{theorem}\cite[Th. 2.4]{CGPT}
If $(\mathfrak{C}, \Gamma)$ is a bichrome diagram obtained as described above from a $3$-dimensional connected oriented closed manifold, then its image $F'(\mathfrak{C}, \Gamma) \in \korps $ depends only on the diffeomorphism class of $M$. 
We can then define the chromatic spherical invariant $\mathcal{K}_{\H-mod}$ by 
$$  \mathcal{K}_{\H-mod} (M) = F'(\mathfrak{C}, \Gamma).  $$
\end{theorem}

\Rq If we cut along the $\alpha$ curves of a handlebody $\mathfrak{C}_\alpha$ associated to a Heegaard diagram $(\Sigma_{\genre}, \alpha, \beta)$, we notice that the multi-handlebody obtained is always connected (and diffeomorphic to $B^3$) because $\Sigma_{\genre} \setminus \alpha$ is connected. \\

The link between $\mathcal{K}_{\H-mod}$ and the map $F''_H$ is established by the following theorem.

\begin{theorem} \label{Th:F''KHmod}
Let $M$ be a $3$-dimensional connected oriented closed manifold and $(\Sigma_{\genre}, \alpha, \beta)$ be a Heegaard diagram of $M$. 
Let $H$ be a spherical Hopf algebra and $\H-mod$ the pivotal category of finite-dimensional $H$-modules. 
Then, for all standardized virtual Heegaard diagrams of the form $(\Sigma_{\genre}, \alpha, \beta, \PtInf, \gamma)$, we have
 $$ \mathcal{K}_{\H-mod}(M) = F''_H(\Sigma_{\genre}, \alpha, \beta, \PtInf, \gamma). $$  
\end{theorem}

\begin{proof}
Let $M$ be a $3$-dimensional connected oriented closed manifold. 
Let us consider a Heegaard splitting $M= \mathfrak{C}_\alpha \cup \mathfrak{C}_\beta$ and an associated Heegaard diagram of $M$, $(\Sigma_{\genre}, \alpha, \beta)$. 
We will assume (potentially after performing an $R_{II}$ move between the $\alpha$ and $\beta$ curves) that, for every component $\alpha_i \in \alpha$, $\alpha_i \cap \beta \not = \emptyset$. 
We fix a point $\PtInf$ in $\Sigma_{\genre} \setminus (\beta \cup \alpha)$ and a family of curves $\gamma \subset \Sigma_{\genre} \setminus \{ \PtInf \}$ such that $(\Sigma_{\genre}, \alpha, \gamma, \PtInf, \gamma)$ is a standardized virtual Heegaard diagram. 
We once again consider a ribbon graph $O_H \subset \Sigma \setminus (\alpha \cup \beta)$ which is the braid closure in a disk of a coupon colored by $h \in \operatorname{End}_{\H-mod}(H)$ such that $F'(O_H) = 1$. 
We choose $O_H$ on $\Sigma$ such that there exists an open disk in $\Sigma_{\genre}$ containing $O_H$, but not containing $\PtInf$, and such that $\gamma \cap O_H = \emptyset$.
Unless stated otherwise, the following operations are skein relations between diagrams, and will be performed in the complement of an open disk containing the coupon colored by $h$. 
We will call this open disk $V_h$. 
Notably, the considered skein relations will always preserve the images under $F'$.

Let us consider, on the pointed surface $\Sigma_{\genre} \setminus \PtInf$, the bichrome diagram formed by the set of $\beta$ curves, considered as red curves, and the graph $O_H$ previously mentioned, considered as the blue subgraph.
Through chromatic modifications with $\tilde{d}$, we turn blue all $\beta$ curves by connecting them with $O_H$ (see Figure \ref{fig:PassBleu}). 
We denote by $(\mathfrak{C}_\alpha, \Gamma)$ the bichrome diagram thus obtained.  
It is always possible (up to isotopy of the surface) to apply successive iterations of the chromatic modification, as shown in Figure \ref{fig:PassBleu}, since the complement of $\Sigma_{\genre} \setminus \beta$ is connected. 

\begin{figure}[ht]
\centering
\includegraphics[scale=0.7]{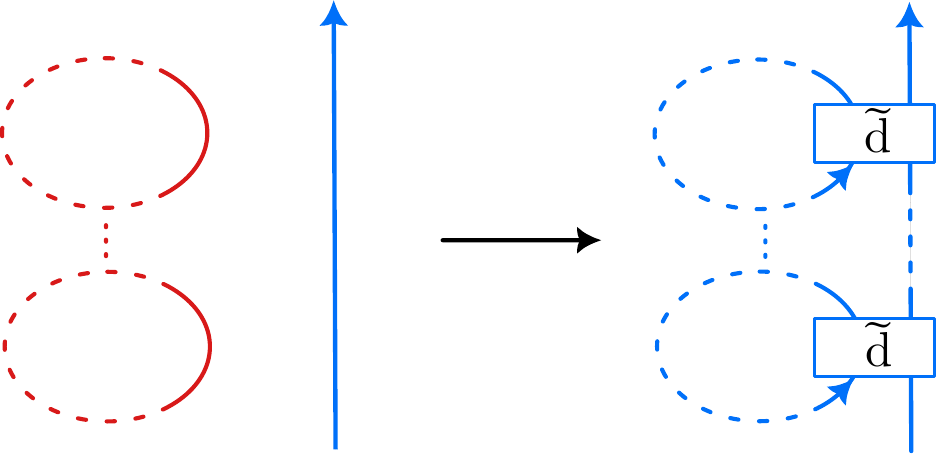}
\caption{Transformation of a bichrome graph into an entirely blue graph with chromatic modifications}\label{fig:PassBleu}
\end{figure}

\vspace{5pt}

We denote $\psi$ a standard projection of $((\Sigma_{\genre}, \alpha, \gamma),\PtInf)$.
We notice that $\Gamma$ and $\beta$ will have the same image by $\psi$ except on an open disk containing the graph $O_H$ and coupons colored by the chromatic morphism.

We use point $4$ of Theorem \ref{Th:DefF'} to cut $(\mathfrak{C}_\alpha, \Gamma)$ along disks $D_i$ with curves $\alpha_i$ as boundaries.
This gives us a new diagram on $B^3$ that we denote $cut(\Gamma)$. 
By definition, $ \mathcal{K}_{\H-mod} (M) = F'(B^3, cut(\Gamma))$.

For all $i \in \{1, ..., {\genre}\}$, cutting along a disk $D_i$ creates two disks in the boundary of $cut_{D_i}(\mathfrak{C}_\alpha)$. 
These disks boundaries, that we will call $\alpha_i^1$ and $\alpha_i^2$, are connected by $\gamma_i$.
Hence, the boundary of a tubular neighborhood $T_i$ of $\alpha_i \cup \gamma_i$ (in $\partial(\mathfrak{C}_\alpha)$) coincides with the boundary of a disk $E_i$ in $\partial(cut_{D_i}(\mathfrak{C}_\alpha))$.
From now on, we consider a projection of $cut(\Gamma)$ onto the plane, identical to the projection $\psi$ previously considered, except on the neighborhoods $T_i$ and the inside of the disks $E_i$ previously mentioned. 
The difference between the images of $T_i$ and $E_i$ by these two projections is represented in Figure \ref{fig:CoupurePlane} (image of $T_i$ on the left and of $E_i$ on the right).

\begin{figure}[ht]
 \centering
 \includegraphics[scale=0.59]{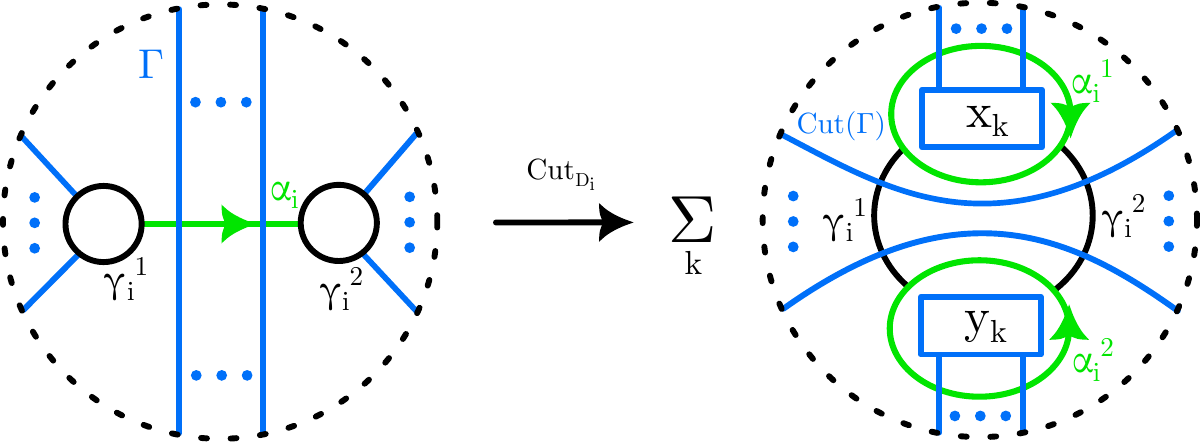}
  \caption{Projections onto the plane before and after cutting along a disk of boundary $\alpha_i$}\label{fig:CoupurePlane}
\end{figure}

For the rest of the computation of $\mathcal{K}_{\H-mod}$, we use graphical calculus in $\Vect-{\korps}$ : all edges are colored by the $\korps$-space $H$, and, for all $h \in H$, the morphism $L_h : x \in H \mapsto hx$ is represented by a bead $h$. 
(i.e. both diagrams represented in Figure \ref{fig:PerleToCoupon} are equivalent.)
\begin{figure}[ht]
\centering
\includegraphics[scale=0.5]{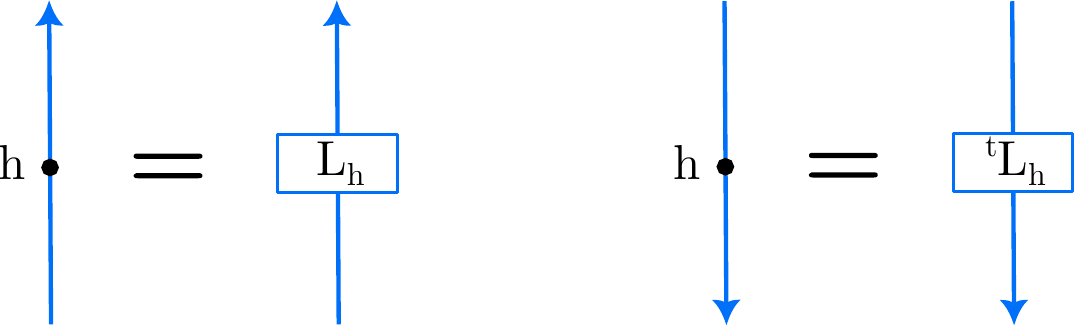}
\caption{Relation between notation with bead and notation with coupons}\label{fig:PerleToCoupon}
\end{figure}
The evaluation and coevaluation morphisms in $\H-mod$ and in $\Vect_{\korps}$ differ by an action of $g^{d}$ with $d \in \{-1, 0, 1\}$.
Hence, we represent on the graph computed on $\Vect$ the beads $g^{\pm 1}$ corresponding to the maxima and minima of the graph (as shown in Figure \ref{fig:PlaceG}).

In the open disk $E_i$, the two coupons of $cut(\Gamma)$ ($x_k$ and $y_k$ in Figure \ref{fig:CoupurePlane}) are separated by the curves that were intersecting $\gamma_i$ before the cutting.
In $\Vect_{\korps}$, Equality $(1)$ shown in Figure \ref{fig:CroisVect} is verified.
This equality is obtained thanks to the braiding $\tau$ of $\Vect$ (defined by $\tau : x \otimes y \in P \otimes P \mapsto y \otimes x$ with $P \in \Proj_H$) represented by crossings in the figure.
This enables us to express the morphism $\sum_k x_k \circ y_k = \Lambda_P$. 

\begin{figure}[ht]
\centering
\includegraphics[scale=0.43]{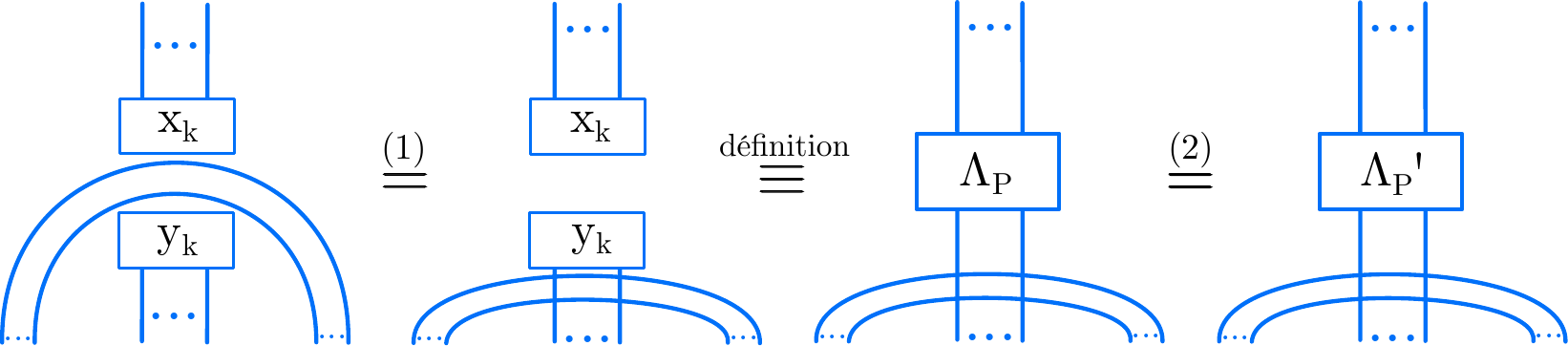}
\caption{Relations between images by $F'$ of diagrams in $\Vect_{\korps}$}\label{fig:CroisVect}
\end{figure}

\begin{lemme}\cite[Lemme 6.2]{CGPT}\label{lem:Coupure}
Let $\Lambda'_P : P \to P$ be the map determined by the action of the cointegral $\Lambda$ over $P \in \Proj_H$ (it is a $H$-module morphism due to $H$ being unimodular, hence $\Lambda$ is central). 
Let $\Lambda_P : P \to P$ be the morphism defined by $\Lambda_P = \sum_k x_k y_k$ with $\{x_k\}$ and $\{y_k\}$ the aforementioned dual bases.
Then $\Lambda_P = \Lambda'_P$ for all projective module $P \in \Proj_H$. 
\end{lemme}

With Lemma \ref{lem:Coupure} we can now "close" the curves where there are coupons $x_k \in \operatorname{Hom}_{\H-mod}(\mathbb{1}, P)$ and $y_k \in \operatorname{Hom}_{\H-mod}(P , \mathbb{1})$ following each other (see Equality $(2)$ in Figure \ref{fig:CroisVect}). 
We will again denote $\Gamma$ the diagram obtained after the "closure" of the curves of $cut(\Gamma)$. 
We now see the cointegral's action $\Lambda \in H$ over the strands appear where the coupon was. 
In other words, on the set of these strands, we now place beads $\Lambda_{(j)}$ or $S(\Lambda_{(j)})$, as indicated in Figure \ref{fig:PerlesKAmod}. 

Since the cutting has been performed along $\alpha$ curves, the resulting graph $\Gamma$ has the same image under the projection onto the plane (therefore the same beads) as a flattening of $\beta$, except inside an open disk $V$ containing the graph $O_H$ and all coupons representing the chromatic morphism. 
From now on, when we will talk about "collected beads" on $\beta$, we will consider the beads obtained by this flattening of $\beta$. 
To end the proof, we will observe the parts of the diagrams contained inside the open disk $V$.

\begin{lemme}\cite[Lemme 6.6]{CGPT}\label{lem:AppChrom}
For all $x \in H$ the following equality is verified in $\Vect_{\korps}$
$$( \overleftarrow{ev}_H \otimes \Id_H)(\Id_{H^*} \otimes L_x \otimes \Id_H)(\Id_{H^*} \otimes \tilde{d})(\overrightarrow{coev}_H \otimes \Id_H) = \lambda(gx) \Id_H,$$
with $\overleftarrow{ev}_H : H^* \otimes H \to \korps$ and $\overrightarrow{coev}_H : \korps \to H^* \otimes H $ the evaluation and coevaluation maps in $\Vect_{\korps}$. 
\end{lemme}

We notice here that the evaluation $\overleftarrow{ev}_H$ and the coevaluation $\overrightarrow{coev}_H$ are considered in $\Vect_{\korps}$ and not with the spherical structure using the pivot $g$. 

\begin{proof}
Let $\{e_i\}_i$ be a basis of $H$ and $\{e_i^*\}_i$ be a dual basis of $H^*$. 
By developing the formulas for the evaluation and coevaluation in $\Vect_{\korps}$ and the formula given for the chromatic morphism, for all $h \in H$ we have 
\begin{align*}
& ( \overleftarrow{ev}_H \otimes \Id_H)(\Id_{H^*} \otimes L_x \otimes \Id_H)(\Id_{H^*} \otimes \tilde{d})(\overrightarrow{coev}_H \otimes \Id_H)(h)  \\
& = \sum_i ( \overleftarrow{ev}_H \otimes \Id_H)(\Id_{H^*} \otimes L_x \otimes \Id_H)(\Id_{H^*} \otimes \tilde{d})(e_i^* \otimes e_i \otimes h) \\
& = \sum_i( \overleftarrow{ev}_H \otimes \Id_H)(\Id_{H^*} \otimes L_x \otimes \Id_H)(e_i^* \otimes \lambda(S( h_{(1)}) g e_i) h_{(2)} \otimes h_{(3)} ) \\
& = \sum_i( \overleftarrow{ev}_H \otimes \Id_H)(e_i^* \otimes \lambda(S( h_{(1)}) g e_i) x h_{(2)} \otimes h_{(3)}) \\
& =  \sum_i \lambda(S( h_{(1)}) g e_i) e_i^*(x h_{(2)}) \otimes h_{(3)})\\
& =  \lambda(S( h_{(1)}) g x h_{(2)}) h_{(3)} \\
& = \lambda(S^2(h_{(2)}) S( h_{(1)}) g x) h_{(3)} \\
& =  \lambda(g x ) \epsilon(h_{(1)}) h_{(2)} \\
& =  \lambda(g x) h.
\end{align*}
 
The third-to-last equality comes from the sphericity of $H$. 
\end{proof}

Since $\lambda (gx) = \mu(x)$ by definition of $\mu$, we can graphically represent in $\Vect_{\korps}$ the lemma by

\begin{center}
\includegraphics[scale=0.5]{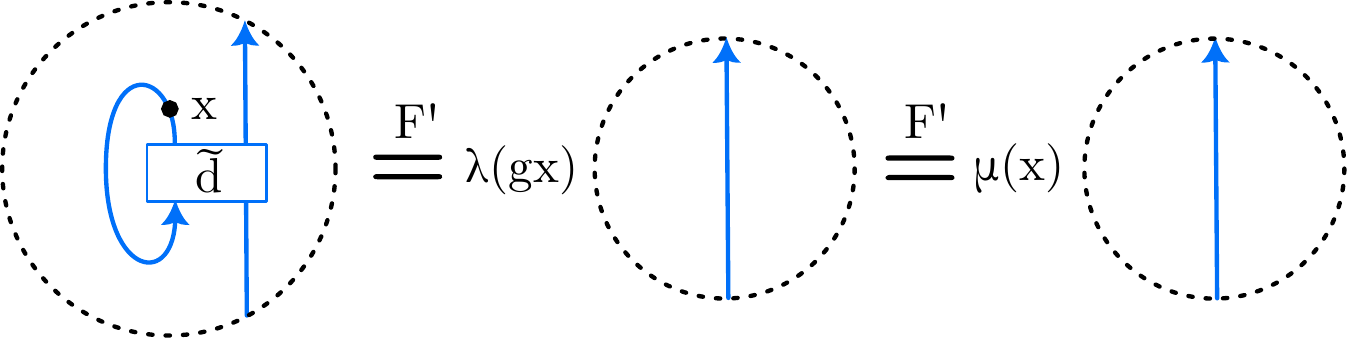} 
\end{center}

The equalities in the figure represent skein relations between images under $F'$ of diagrams differing only on an open set as indicated.

Hence, for each coupon colored by the chromatic morphism on $\Gamma$, we obtain a scalar $\mu(\perleTot_i)$ with, for all $i \in \{1, ..., {\genre}\}$, the product of the component's beads denoted $\perleTot_i \in H$. 
This product $\perleTot_i \in H$ is identical to the total bead $\perleTot_i$ obtained on the corresponding component of $\beta$ during the computation of $F''$ for a base point placed inside $V$.
Indeed, we have seen previously that, outside the open set $V$, $\Gamma$ has the same components and beads as those found on $\beta$ during the computation $F''$. 

After applying Lemma \ref{lem:AppChrom} on each coupon colored by the chromatic morphism, we note that the image of graph $\Gamma$ under $F'$ is reduced to the image of graph $O_H$ under $F'$ multiplied by scalars $\mu(\perleTot_i)$. 
By definition of $O_H$ and of $F'$, we find that $F'(\mathfrak{C}_\alpha, \Gamma) = \prod_{i \in \{1, ..., {\genre}\}} \mu(\perleTot_i) F'(B^3,O_H) = \prod_{i \in \{1, ..., {\genre}\}} \mu(\perleTot_i) = F''_H(\Sigma_{\genre}, \alpha, \beta, \PtInf, \gamma) $. 
Which gives us the desired result. 

\end{proof}

Since $\mathcal{K}_{\H-mod}$ does not depend on the placement of point $\PtInf$, nor on the $\gamma$ curves, Theorem \ref{Th:F''KHmod} gives the following corollary. 

\begin{coro}
Let $M$ be a $3$-dimensional connected oriented closed manifold and $(\Sigma_{\genre}, \alpha, \beta)$ be a Heegaard diagram of $M$.
Let $\PtInf$ be a point and $\gamma$ be curves such that $(\Sigma_{\genre}, \alpha, \beta, \PtInf, \gamma)$ is a standardized virtual Heegaard diagram.
Let $H$ be a spherical Hopf algebra and $\H-mod$ the pivotal category of finite-dimensional $H$-modules.  
The scalar $F''(\Sigma_{\genre}, \alpha, \beta, \PtInf, \gamma)$ does not depend on the choices of the point $\PtInf$ or the curves $\gamma$. 
\end{coro} 

Consequently, for the rest of the article, we will allow ourselves to omit arguments $\PtInf$ or $\gamma$ in the map $F''$ when $(\Sigma_{\genre}, \alpha, \beta)$ is a non-virtual Heegaard diagram or a flattening of a non-virtual Heegaard diagram.

\section{Relation between invariants}

In this section, we fix $H$ a spherical Hopf algebra on a field $\korps$.

This section rests on a correspondence between specific Heegaard and surgery link diagrams. 
This approach is used and detailed in \cite[Ch. 4]{Chang_Cui} and \cite[Ch. 3.1]{Chang_Wang}.

\begin{theorem} \cite[Th. 4.2]{Chang_Cui}\cite[Th. 4.1 \& Th.4.2]{BirPow}\label{ExDiagStand}
Any $3$-dimensional oriented closed manifold $M$ admits a Heegaard diagram $D=(\Sigma_{\genre}, \alpha, \beta)$, with some genus ${\genre}$, verifying the following properties: 
\begin{itemize}
\item There exists a collection of ${\genre}$ curves $\gamma = \{\gamma_1, ... , \gamma_{\genre}\}$ on $\Sigma_{\genre}$ such that $D_1 =  (\Sigma_{\genre}, \gamma, \beta)$ and $D_2 =(\Sigma_{\genre}, \alpha, \gamma)$ are diagrams of $S^3$ verifying $\MatInter{\gamma}{\beta} = \Id_{\genre}$ and $\MatInter{\alpha}{\gamma} = \Id_{\genre}$. 
\item If we see $\beta$ as a framed link in $\mathbb{S}^3$, where $\mathbb{S}^3$ is determined by $D_1$, 
and where framing is chosen as a parallel copy of $\beta$ in the Heegaard surface, then $\beta$ is a surgery link of $M$. 
\end{itemize}

\end{theorem}
The idea behind the proof in \cite[Theorem 4.1 and Theorem 4.2]{BirPow} is to use a surgery presentation of $M$ by a link $L \subset \mathbb{R}^2 \times \mathbb{R}^{-} \subset \mathbb{S}^3$ that we will consider as the flat closure of a pure braid. 
We can then "untangle" that link in such a way that the only strands of $L$ not resting in the plane $\mathbb{R}^2 \times \{0\}$ are unknotted arcs with trivial framing. 

\vspace{5pt}
We will now describe a way to see the $\beta$ curves of a Heegaard diagram verifying the hypotheses of Theorem \ref{ExDiagStand} as a link in $\mathbb{R}^3$. 
This perspective is based on the reasoning inside the proof mentioned above, which is why the link obtained will be the one mentioned in the theorem. 

Let $(\Sigma_{\genre}, \alpha, \beta)$ be a Heegaard diagram of a ($3$-dimensional oriented closed) manifold $M$ and $\gamma \subset \Sigma_{\genre}$ be a collection of curves such as described in Theorem \ref{ExDiagStand}. 
Therefore, there exists a point $\PtInf \in \Sigma_{\genre} \setminus (\alpha \cup \beta)$ such that $(\Sigma_{\genre}, \alpha, \beta, \PtInf, \gamma)$ is a standardized virtual Heegaard diagram, with changes of the orientation of certain curves of $\gamma$ if necessary. 
Let us consider $\psi$ a standard projection of $(\Sigma_{\genre}, \alpha, \beta, \PtInf, \gamma)$. 
We will see a way to visualize the image $\psi(\beta) \subset \mathbb{R}^2$ as a link in $\mathbb{R}^3$ (the following explanations are illustrated in Figure \ref{fig:Anse3D}). 
Let $r_\theta$ be the rotation with radius $\theta$ around axis $\{\frac{3}{2}\} \times \mathbb{R} \times \{0\}$ (counterclockwise). 
For all $i \in \{1, ..., {\genre}\}$, the image $r_{[0, \pi]}(D_i^1) = \{r_{\theta}(D_i^1) | \theta \in [0, \pi]\}$ describes an unknotted handle with trivial framing and attaching region $\{D_i^1 \cup D_i^2 \}$, attached to $\mathbb{R}^2 \times \mathbb{R}^{+}$.
The image of $\psi(\beta) \cap \gamma_i^1$ by the set of rotations $r_{[0, \pi]}$ is a strand in the boundary of the handle thus defined, which connects $\psi(\beta) \cap \gamma_i^1$ and $\psi(\beta) \cap \gamma_i^2$. 
The strands thus defined correspond to the unknotted arcs with trivial framing mentioned in the idea behind the proof of Theorem \ref{ExDiagStand}.
Therefore, we obtain a link $ \psi(\beta) \bigcup_{1 \leq i \leq {\genre}} r_{[0, \pi]}(\psi(\beta) \cap \gamma_i^1)$ in $\mathbb{R}^3 (\subset \mathbb{S}^3)$ which is, in particular, a surgery link of the manifold $M$ (see proof of \cite[Theorem 4.1 and Theorem 4.2]{BirPow}). 

\begin{figure}[ht]
\centering
\includegraphics[scale=1.3]{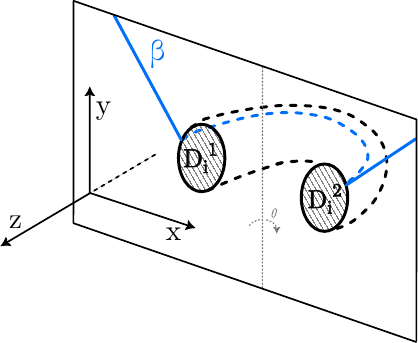}
\caption{Visualization of the image of a standard projection in $\mathbb{R}^3$}\label{fig:Anse3D}
\end{figure}

\Rq In the rest of this article, we will represent the image of a Heegaard diagram $D \subset \mathbb{R}^2 \subset \mathbb{R}^3$ by a standard projection on $\mathbb{R}^2$, and an associated surgery link $L \subset \mathbb{R}^3 (\subset \mathbb{S}^3)$ as shown in Figure \ref{fig:CompHeegEntrel}.
Outside the open disks represented in the figure, it follows from the above that the link $L$ can be projected onto the plane (with the same standard projection), and we can note that we obtain the same curves for the projection of the Heegaard diagram and for the associated surgery link seen in the plane. 
\vspace{5pt}

\begin{figure}[ht]
\centering
\includegraphics[scale=0.65]{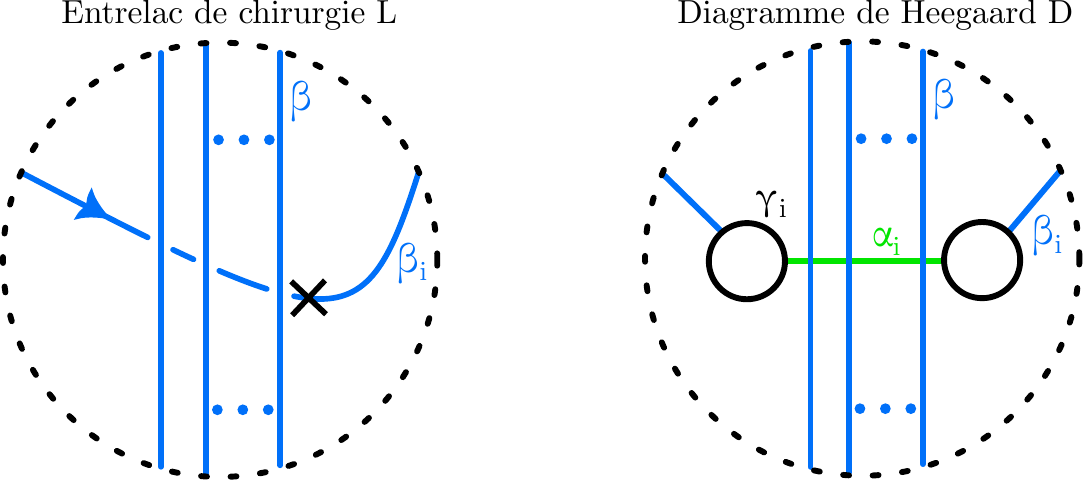}
\caption{Heegaard diagram and surgery link of the same manifold $M$}\label{fig:CompHeegEntrel}
\end{figure}

This similarity between diagrams allows us to see the following relation between the spherical chromatic invariant $\mathcal{K}_{\H-mod}$ and the renormalized Hennings-Kauffman-Radford invariant $\HKR_{D(H)}$ of $M$.

\begin{theorem}[Main theorem]\label{Th:principal}
Let $H$ be a spherical Hopf algebra and $M$ be a $3$-dimensional connected oriented closed manifold. 
We have
$$  \mathcal{K}_{\H-mod}(M) = \HKR_{D(H)}(M) $$
with $D(H)$ the Drinfeld double of $H$. 
\end{theorem}

\Rq We suspect to be true the following, more general, version of this theorem, 
$$ \mathcal{K}_{\categ}(M) = \HKR_{Z(\categ)}(M) $$
with $\categ$ a spherical category and $Z(\categ)$ its Drinfeld center. 

\begin{proof}
Let $H$ be a spherical Hopf algebra and $M$ be a $3$-dimensional connected oriented closed manifold. 

By Theorem \ref{ExDiagStand}, there exists a Heegaard diagram and a surgery link representing the manifold $M$ and respecting the assumptions of that theorem (we have seen that these diagrams will have analogous forms). 

We can assume (potentially by adding $R_{II}$ moves between $\alpha$ and $\beta$ curves) that, for all $i \in \{1, ..., {\genre}\}$, the intersection $\alpha_i \cap \beta$ is non-empty. 

Therefore, on one hand, we have $D=(\Sigma_{\genre}, \alpha, \beta)$, a Heegaard diagram of $M$ along with a point $\PtInf$ and a collection of curves $\gamma$ such that $(\Sigma_{\genre}, \alpha, \beta, \PtInf, \gamma)$ is standardized, and that $\MatInter{\gamma}{\beta} = \Id_{\genre}$.
We will compute the bichrome invariant $\mathcal{K}_{\H-mod}(M)$ with this diagram. 
On the other hand, we have $L$, a surgery link of $M$ originating from the $\beta$ curves. 
We will compute the invariant $\HKR_{D(H)} (M)$ with that link.

Each component $\beta_i$ of $L$, $i \in \{1, ..., {\genre}\}$, contains a "horizontal" segment included in any neighborhood (on the plane) of the curve $\alpha_i$ (because $\MatInter{\gamma}{\beta} = \Id_{\genre}$ and $\MatInter{\gamma}{\alpha} = \Id_{\genre}$). 
Hence, we can consider that each component $\beta_i$ is oriented in such a way that this horizontal segment of $\beta_i$ will be oriented from left to right.
Indeed, if it is not the case for a given curve $\beta_i$, we can reverse its orientation without changing the image of the invariants, since $\HKR$ is invariant under orientation reversal of any curve $\beta_i$ (the same is true for $\mathcal{K}$ by Proposition \ref{prop:ChangOrientation}). 

On each component $\beta_i$, we will choose to place the base points on the right of all crossings of this horizontal strand, and before any extremum on $\beta_i$ (see the left in Figure \ref{fig:CompHeegEntrel}).

We will seek to make explicit $\HKR_{D(H)}(M)$. 
We know, by results of Part \ref{PartDoubleDrinfeld}, that we can express elements of $D(H)$ with those of $H$ and $H^*$. 
Hence, Lemma \ref{MultiPerles} will allow us to express $\HKR_{D(H)}(M)$ with elements of $H$ (and $H^*$).

\vspace{10pt}

We start by paying attention to the beads placed on crossings during the computation of the $\HKR$ invariant. 
Each crossing gives two beads that will be collected in different orders during their multiplication. 

\vspace{5pt}
\Rq We denote the $R$-matrix of $D(H)$ by $R^D = \sum_i r^D_i \otimes s^D_i$ (or by $R^D = r^D_i \otimes s^D_i$, implying the summation). 
Therefore, we will talk about an $s$ bead (resp. an $r$ bead) for beads colored by elements $s^D_i$ (resp. $S^n(r^D_i)$ with $n \in \mathbb{Z}$) of the $R$-matrix (see Figure \ref{fig:FoncPerleGPP}).  
We will also talk about $g$ beads for beads colored by $g \in H$, $g^D \in D(H)$ or their inverse. 

\vspace{5pt}

What follows can be visualized on the left side of Figure \ref{fig:CompEntrHeeg}.

We notice that $s$ beads will always follow each other on that horizontal segment, without $g$ bead in between. 
Conversely, $r$ beads will always be on "vertical" segments and can have $g$ beads between them.

For each crossing, we associate a pair of indices $(i,h) \in \mathbb{Z}^2$. 
Index $i$ is the index of the undercrossing component. 
Here, it is also the horizontal component.  
If we order the crossings from the base point, following orientation on $\beta_i$, this crossing is the $h$-th one among the undercrossings of component $\beta_i$. 
We denote by $k_i$ the number of under-crossings of component $\beta_i$ and ${\genre}$ is the total number of components (we recall that ${\genre} \in \mathbb{N}$ is also the genus of the associated Heegaard diagram of $M$ determined in Theorem \ref{ExDiagStand}). 

We will denote $s^D_{i,h}$ and $r^D_{i,h}$ the elements coloring the beads of the crossing $(i,h)$, and $d_{i,h} \in \mathbb{Z}$ the exponent of the antipode applied to $r^D_{i,h}$. 

We recall that, by definition (see Figure \ref{fig:FoncPerleGPP}), $d_{i,h}$ depends on the orientation of $\beta$ at the crossing. 
Hence, we have
\begin{equation*}
      d_{i,h} =
     \begin{cases}
     0 & \text{ if the vertical } \beta \text{ curve is downward},\\
     -1 & \text{else.}
     \end{cases}
\end{equation*}

\begin{figure}[ht]
\centering
\includegraphics[scale=0.65]{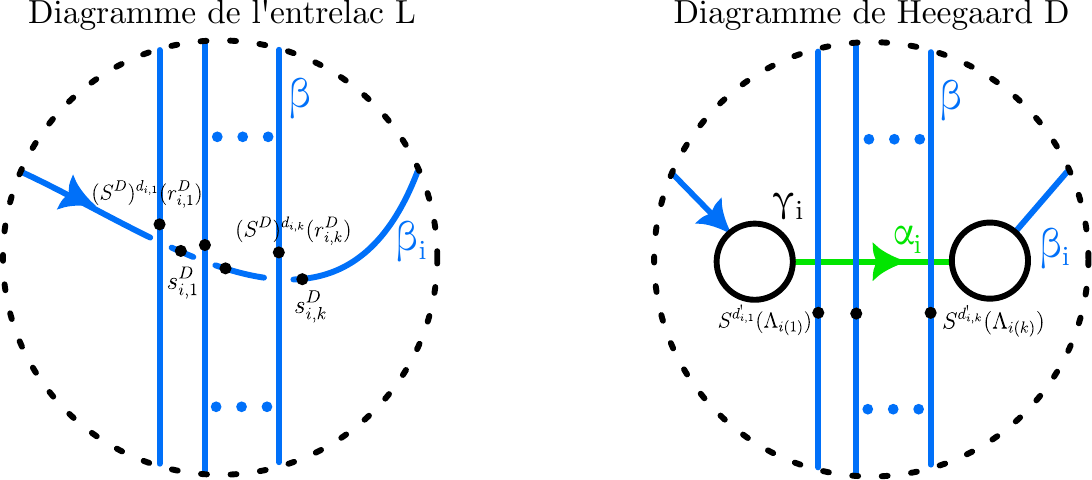}
\caption{Beads from crossings on $L$ and on $D$}\label{fig:CompEntrHeeg}
\end{figure}

\vspace{5pt}

For each crossing, we can associate a second pair of indices $(j,v) \in \mathbb{N}^2$ with $j$ the index of the overcrossing component (which here is also the vertical component), and, if we order overcrossings from the base point, following orientation on the component, this crossing is the $v$-th among the overcrossings of component $j$. 
We denote by $l_j$ the number of overcrossings of the component $j$. 
This allows us to define the map $\phi$ by $\phi(j,v) = (i,h)$. 
That map associates the component and collection order of an $r$ bead to the component and collection order of the $s$ bead from the same crossing. 
For example, in the case represented in Figure \ref{fig:ExempleEntrelacs}, we will have $\phi(1,1) = (2,1), \phi(1,2) = (1,1), \phi(2,1) = (2,3), \phi(2,2) = (2,2)$ and $\phi(2,3) = (1,2)$. 
We will sometimes denote $\phi(j,v) = (p_{j,v} , q_{j,v})$ when we will need to refer independently to the components of the pair. 

\Rq Since the studied Hopf algebras are non-commutative, we will denote 
$$ \prod^\longrightarrow_{m \in (n, ..., 1)} a_i = a_n a_{n-1} \cdots a_1 $$
to stress the order of the elements in a product. 
\vspace{5pt}

The total bead collected on each component $i$ during the computation of the $\HKR$ invariant is 
$$s^D_{i,k_i} \cdots  s^D_{i,1} \Big(\prod^\longrightarrow_{m \in (l_i, ..., 1)} (g^D)^{n_{i,m}} (S^D)^{d_{\phi(i,m)}}(r^D_{\phi(i,m)}) \Big)(g^D)^{n_{i,0}}$$
with $n_{i,m}$ the algebraic sum of the exponents of all $g$ beads between the bead colored by $r^D_{\phi(i,m)}$ and the next $r$ or $s$ bead (or between the base point and the first $r$ bead in the case of $n_{i,0}$). 
See Figures \ref{fig:ExempleDiagramme} and \ref{fig:ExempleEntrelacs} for an example illustrating this formula. 

With the beads thus collected, the formula of $\HKR$ gives 
$$\HKR_{D(H)}(M) =$$
$$ (\delta^D)^{-s} \prod_{i=1}^{{\genre}} \mu^D \Big(s^D_{i,k_i} \cdots  s^D_{i,1} \big(\prod^\longrightarrow_{m \in (l_i, ..., 1)} (g^D)^{n_{i,m}} (S^D)^{d_{\phi(i,m)}}(r^D_{\phi(i,m)}) \big)(g^D)^{n_{i,0}}\Big)$$
where $s$ is the signature of the linking matrix of the link. 
We will show in Lemma \ref{MultiPerles} that this expression can be reformulated with elements from $H$, which gives us 

$$\HKR_{D(H)}(M) = \delta^{-s} \prod_{i = 1}^{\genre} \mu \Big( \big(\prod^\longrightarrow_{m \in (l_i, ..., 1)} g^{n_{i,m}} S^{d_{\phi(i,m)}+1}(\Lambda_{p_{i,m},(q_{i,m})}) \big) g^{n_{i,0}} \Big).$$

Let us now return to the computation of $\mathcal{K}_{\H-mod}(M)$. 
During the computation of this invariant, a single bead is associated to each crossing between $\alpha_i$ and $\beta_j$ (see Figure \ref{fig:PerlesKAmod}). 

Due to the correspondence between the crossing of the Heegaard diagram $D$ and those of the surgery link $L$, we can use the map $\phi$ as seen above. 
For each crossing, the map $\phi$ associates to a pair of indices $(j,v)$ the pair $(i,h)$ with $j,i,v$ and $h$ integers defined as follows. 
Here, $j$ is the index of the component $\beta_j$ on which the bead will be placed and $i$ is the index of the component $\alpha_i$ involved in the crossing. 
If we order the beads along $\beta_j$, starting at the base point and following the orientation, then this bead is the $v$-th one collected on $\beta_j$. 
If we order the crossings on $\alpha_i$ from left to right, then this crossing is the $h$-th one among the crossings on $\alpha_i$. 

The total bead collected on each component $i$ in the Heegaard diagram $D$ is given by 
$$ g^{n'_{i,l_i}} S^{d_{\phi(i,l_i)}+1}(\Lambda_{\phi(i,l_i)}) \cdots S^{d_{\phi(i,1)}+1}(\Lambda_{\phi(i,1)}) g^{n'_{i,0}}, $$
with $d_{i,m} \in \{-1,0\}$ and ${n_{i,m}}$ the exponents previously seen for all $i \in \{1, ..., {\genre} \}$ and all $m \in \{1, ..., l_i\}$. 
Indeed, by definition of beads on crossings (see Figures \ref{fig:FoncPerleGPP} and \ref{fig:PerlesKAmod}), we note that we recover the exponents $d_{i,m}+1$ on the antipodes. 
Furthermore, since we chose analogous forms for $L$ and $D$, we know that we have the same minima, maxima, and crossings (in amount and orders) on both.
Hence, $n_{i,m}$ is also the algebraic sum of exponents on $g$ beads between the bead colored by $S^{d_{\phi(i,m)}+1}(\Lambda_{\phi(i,m)})$ (or the base point if $m=0$) and the next bead that is not a $g$ bead (or the base point if $m=l_i$).  

Therefore, by Theorem \ref{Th:F''KHmod}, we find the following formula for $\mathcal{K}_{\H-mod}(M)$. 
$$ \mathcal{K}_{\H-mod}(M) = \epsilon(\Lambda)^v \prod_{i = 1}^{\genre} \mu (\Big( \big(\prod^\longrightarrow_{m \in (l_i, ..., 1)}g^{n_{i,m}} S^{d_{\phi(i,m)}+1}(\Lambda_{\phi(i,m)}) \big)g^{n_{i,0}}\Big),$$
where $v = 0$, hence $\epsilon(\Lambda)^v =1$ because we assumed earlier that any $\alpha$ curve intersects the set $\beta$. 
See Figures \ref{fig:ExempleDiagramme} and \ref{fig:ExempleDiagrammePlat} for an example illustrating this formula.

If we apply Lemma \ref{MultiPerles} to the expressions we have for the invariants, we can conclude that 
$$  \HKR_{D(H)}(L) =  (\delta^D)^{-s} \mathcal{K}_{\H-mod}(D).$$

Hence, proving that $\delta^D=1$ (which will be done in Lemma \ref{lem:delta=1}) allows us to directly conclude the expected result $\HKR_{D(H)}(M) =  \mathcal{K}_{\H-mod}(M) $. 

To finish the proof, all that is left is to prove the two lemmas announced above.

\begin{lemme}\label{MultiPerles}
Let us consider a spherical algebra $H$ and its double $D(H)$ with the notations previously introduced and from Part \ref{PartDoubleDrinfeld}. 

We then have the equality
$$ \prod_{b = 1}^{\genre} \mu^D \Big(s^D_{i_{b,k_b}} \cdots s^D_{i_{b,1}} \big(\prod^\longrightarrow_{m \in (l_b, ..., 1)} (g^D)^{n_{b,m}}  (S^D)^{d_{\phi(b,m)}}(r^D_{i_{\phi(b,m)}}) \big) (g^D)^{n_{b,0}} \Big) $$
$$ = \prod_{b = 1}^{\genre} \mu \Big( \big(\prod^\longrightarrow_{m \in (l_b, ..., 1)} g^{n_{b,m}} S^{d_{\phi(b,m)}+1}(\Lambda_{p_{b,m},(q_{b,m})}) \big) g^{n_{b,0}} \Big), $$

with, for all $(i,j) \in \mathbb{N}^2$, $\Lambda_{p_{i,j}, (q_{i,j})}$ the $q_{i,j}$-nth factor in the coproduct $\Delta^{(l_{p_{i,j}})}(\Lambda_{p_{i,j}})$ (cointegral associated to the component ${p_{i,j}}$). 

\end{lemme}

\begin{proof}
Let $(e_1, ..., e_n)$ be a basis of $H$ and $(e^*_1, ..., e^*_n)$ be a dual basis of $H^{*,\mathrm{cop}}$, we recall that we may make explicit the $R$-matrix of $D(H)$ by
$$ R^D := \sum^n_{i = 1} (\epsilon \otimes e_i) \otimes (e^*_i \otimes 1_H). $$
Thus, we have
\begin{align*}
&  \prod_{b = 1}^{\genre} \mu^D \Big(s^D_{i_{b,k_b}} \cdots s^D_{i_{b,1}} \big(\prod^\longrightarrow_{m \in (l_b, ..., 1)} (g^D)^{n_{b,m}} (S^D)^{d_{\phi(b,m)}}(r^D_{i_{\phi(b,m)}}) \big)(g^D)^{n_{b,0}} \Big)  \\ 
&\overset{(1)}= \prod_{b = 1}^{\genre} (\Lambda \otimes \mu) \Big((e^*_{i_{b,k_b}} \otimes 1_H) \cdots (e^*_{i_{b,1}} \otimes 1_H) \\
&\qquad \big(\prod^\longrightarrow_{m \in (l_b, ..., 1)} (\epsilon \otimes g)^{n_{b,m}}(\epsilon \otimes S^{d_{\phi(b,m)}}(e_{i_{\phi(b,m)}})) \big) (\epsilon \otimes g)^{n_{b,0}} \Big) \\
&\overset{(2)}= \prod_{b = 1}^{\genre} (\Lambda \otimes \mu) \Big(e^*_{i_{b,k_b}} \cdots e^*_{i_{b,1}} \otimes  \big(\prod^\longrightarrow_{m \in (l_b, ..., 1)} g^{n_{b,m}}S^{d_{\phi(b,m)}}( e_{i_{\phi(b,m)}}) \big) g^{n_{b,1}} \Big) \\
&\overset{(3)}= \prod_{b = 1}^{\genre} e^*_{i_{b,k_b}}(\Lambda_{b,(1)}) \cdots e^*_{i_{b,1}}(\Lambda_{b,(k_b)})  \mu \Big( \big(\prod^\longrightarrow_{m \in (l_b, ..., 1)} g^{n_{b,m}} S^{d_{\phi(b,m)}}( e_{i_{\phi(b,m)}})\big) g^{n_{b,0}} \Big) \\
&\overset{(4)}= \prod_{b = 1}^{\genre} \mu \Big( \big(\prod^\longrightarrow_{m \in (l_b, ..., 1)} g^{n_{b,m}}S^{d_{\phi(b,m)}}( \Lambda_{p_{b,m}, (l_b-q_{b,m} +1)}) \big) g^{n_{b,0}} \Big) \\
&\overset{(5)}= \prod_{b = 1}^{\genre} \mu \Big( \big(\prod^\longrightarrow_{m \in (l_b, ..., 1)} g^{n_{b,m}}S^{d_{\phi(b,m)}+1}( \Lambda_{p_{b,m}, (q_{b,m})})\big) g^{n_{b,0}} \Big).
\end{align*}

Equality $(1)$ is obtained by expressing the elements of $D(H)$ with elements of $H$ according to the formulas from Part \ref{PartDoubleDrinfeld}. 

Due to the expression for multiplication in the double, for all $f, f' \in H^{*,\mathrm{cop}}$, and all $x, x' \in H$, we have 
$$ (f \otimes x)(f'\otimes x') = ff' \otimes xx' \qquad \text{ if } x = 1_H \text{ or } f'= \epsilon. $$
Which gives equality $(2)$.  

Equality $(3)$ is the application of $\Lambda \otimes \mu$ to elements of $D(H)$. 
We distinguish here $g$ copies of $\Lambda$, a copy $\Lambda_i$ being associated to each component $\beta_i$. 

By a property of the dual basis, for all $x \in H$, and $f \in H^*$ then $\sum_i e_i^*(x) f(e_i) = f(x)$, which gives equality $(4)$. 

Equality $(5)$ is obtained by the relation $S(\Lambda) = \Lambda$. 

\end{proof}

\begin{lemme}\label{lem:delta=1}
Let $H$ be a spherical Hopf algebra with pivot $g$ and $D(H)$ be its Drinfeld double with the notations introduced in Part \ref{PartDoubleDrinfeld}. 
Then $\delta^D := \mu^D(g^D \theta^D) = 1$, with $\theta^D := \sum_i s^D_i g^D r^D_i$. 
\end{lemme}

\begin{proof}
Since $H$ is spherical, we can apply Proposition \ref{g_D} and express elements of $D(H)$ with elements of $H$. 
\begin{align*}
\delta^D & =  \mu^D(g^D \theta^D) =  \mu^D(g^D \sum_i s_i^D g^D r_i^D) = \sum_i \mu^D(g^D s_i^D g^D r_i^D) &&\\
	& = \sum_i (\Lambda \otimes \mu) ((e_i^* \otimes 1_H) (\epsilon \otimes g )(\epsilon \otimes e_i) (\epsilon \otimes g)) &&\\
	& =  \sum_i (\Lambda \otimes \mu) (e_i^* \otimes g e_i g) &&\\
	& = \sum_i e_i^*(\Lambda) \otimes \mu (g e_i g) &&\\
	& = \mu (g^2 \Lambda) = \epsilon(g)^3 \lambda(\Lambda) = 1 &&
\end{align*}
\end{proof}

\Rq Here it is essential for $D(H)$ to be ribbon and unimodular (because quasitriangular and spherical) to be able to express $g^D$ with $g$.  

 \vspace{5pt}
This concludes the proof of Theorem \ref{Th:principal}.

\end{proof}

\textit{Example : }

\begin{figure}[ht]
\centering
\includegraphics[scale=0.12]{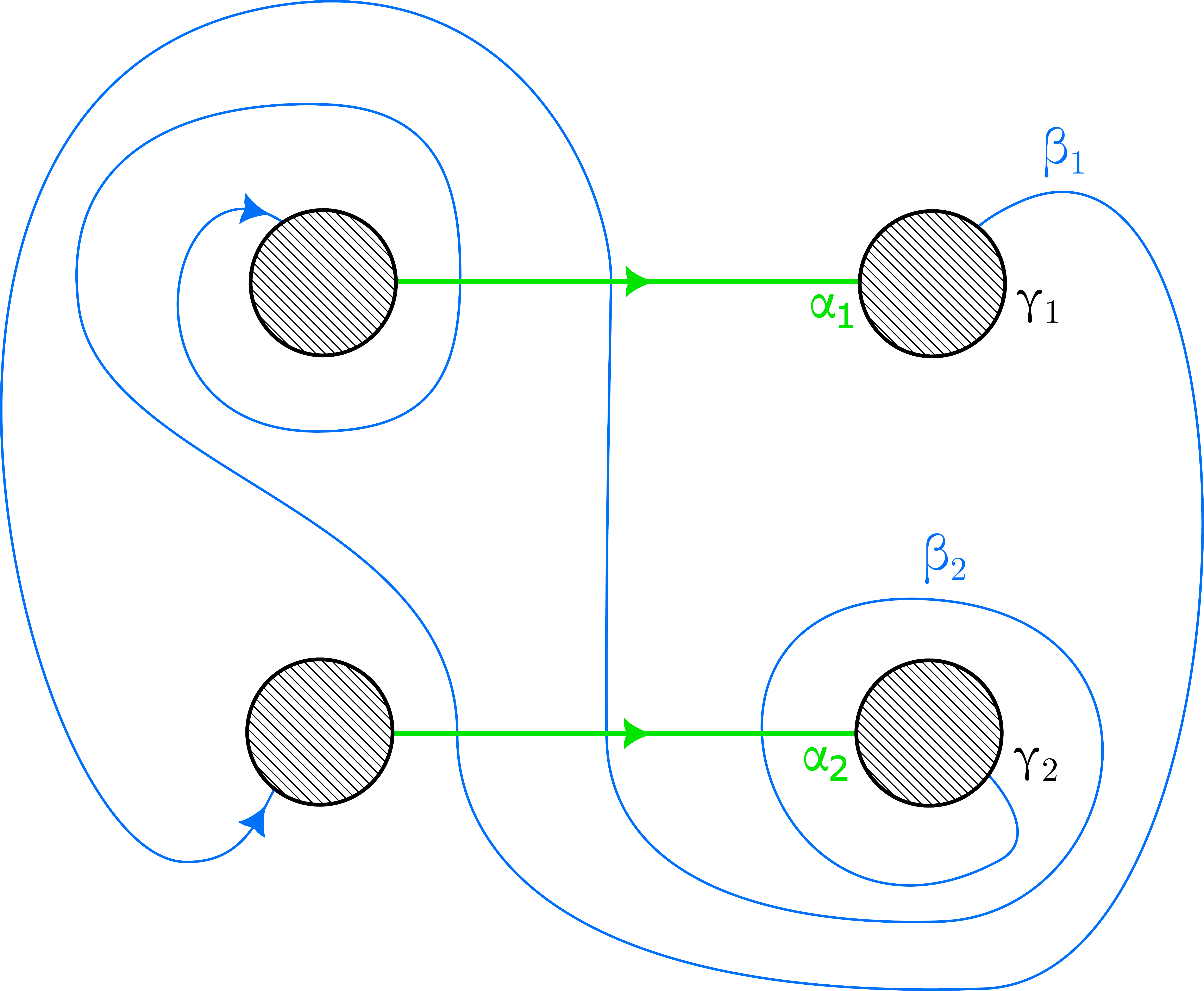}
\caption{Heegaard diagram of the manifold $\mathbb{S}^3$}\label{fig:ExempleDiagramme}
\end{figure}

Applied to the diagram in Figure \ref{fig:ExempleDiagramme}, the process described in part \ref{Part:CalcHenn} to compute the invariant $\HKR_{D(H)}$ (see Figure \ref{fig:ExempleEntrelacs}) gives the scalar 
\begin{align*}
S_{\HKR} =& \mu^D(s^D_{1,2} s^D_{1,1} g^D r^D_{1,1} g^D (S^D)^{-1}(r^D_{2,1}) )\\
& \mu(s^D_{2,3} s^D_{2,2} s^D_{2,1} (S^D)^{-1}(r^D_{1,2}) (S^D)^{-1}(r^D_{2,2}) g^D (S^D)^{-1}(r^D_{2,3})) \qquad \in \korps.
\end{align*}

\begin{figure}[ht]
\centering
\includegraphics[scale=0.11]{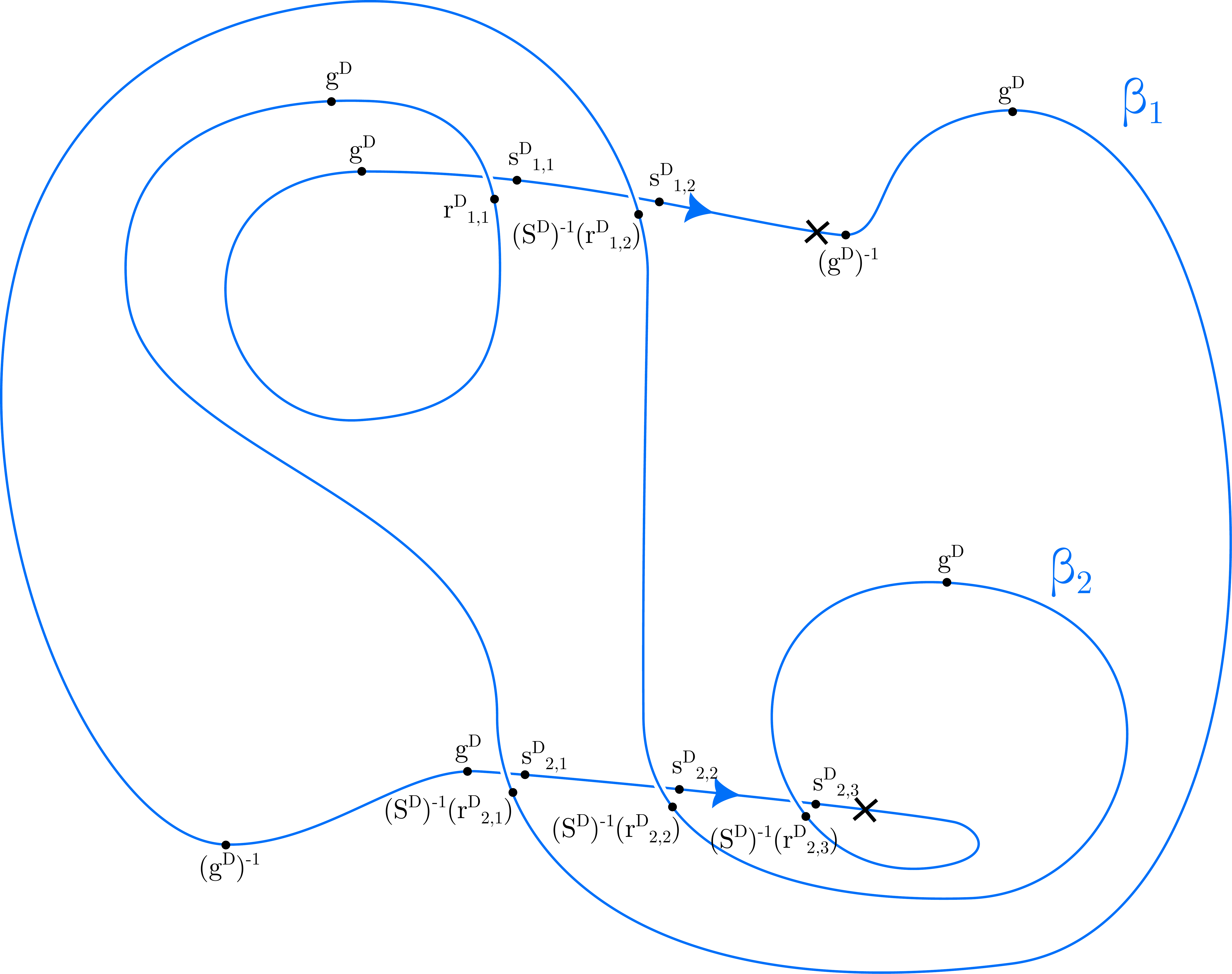}
\caption{Surgery link associated to the virtual diagram represented in Figure \ref{fig:ExempleDiagrammePlat}}\label{fig:ExempleEntrelacs}
\end{figure}

Applied to the diagram in Figure \ref{fig:ExempleDiagramme}, the process described in part \ref{Part:F''} to compute the invariant $\mathcal{K}_{\H-mod}$ (see Figure \ref{fig:ExempleDiagrammePlat}) gives the scalar
$$ S_\mathcal{K} = \mu(g S(\Lambda_{1,(1)}) g \Lambda_{2,(1)}) \mu(\Lambda_{1,(2)} \Lambda_{2,(2)} g \Lambda_{2,(3)}) \qquad \in \korps.$$

\begin{figure}[ht]
\centering
\includegraphics[scale=0.11]{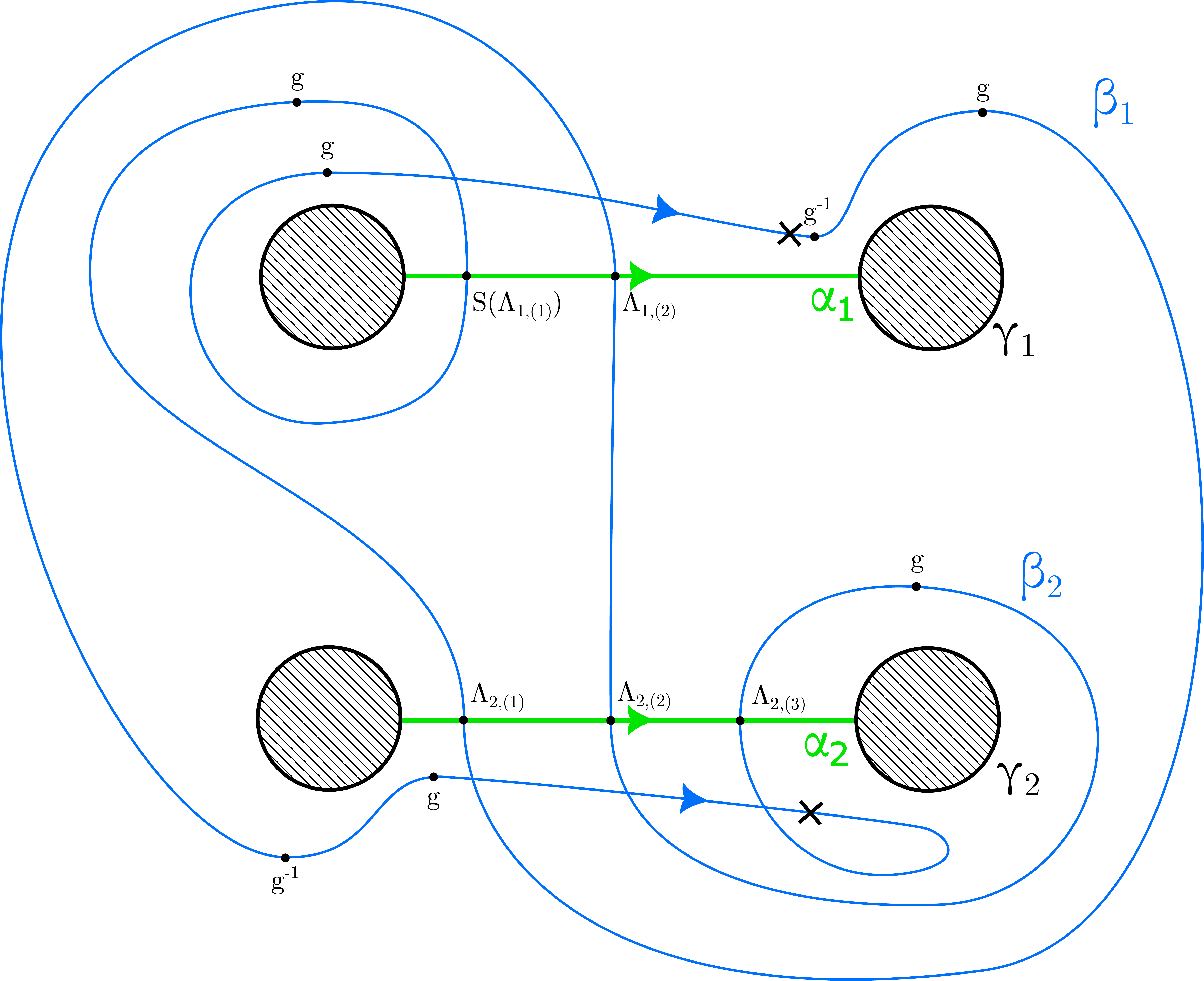}
\caption{Flattening of the diagram represented in Figure \ref{fig:ExempleDiagramme}}\label{fig:ExempleDiagrammePlat}
\end{figure}

\clearpage

\bibliography{Biblio}

\bibliographystyle{alpha}

\end{document}